\documentclass{article}

\usepackage[english]{babel}

\usepackage[letterpaper,top=2cm,bottom=2cm,left=3cm,right=3cm,marginparwidth=1.75cm]{geometry}

\usepackage{amsmath}
\usepackage{amssymb}
\usepackage{amsthm}
\usepackage{graphicx}
\usepackage[colorlinks=true, allcolors=blue]{hyperref}
\usepackage{tikz}
\usepackage{tikz-3dplot}

\newtheorem{theorem}{Theorem}
\newtheorem{lemma}[theorem]{Lemma}
\newtheorem{corollary}[theorem]{Corollary}
\newtheorem{proposition}[theorem]{Proposition}

\newtheorem{conjecture}[theorem]{Conjecture}

\theoremstyle{definition}
\newtheorem{definition}[theorem]{Definition}
\newtheorem{example}[theorem]{Example}
\newtheorem{remark}[theorem]{Remark}

\title{More on the optimal arrangement of $2d$ lines in $\mathbb{C}^d$}
\author{Joseph W.\ Iverson\thanks{Department of Mathematics, Iowa State University, Ames, Iowa, USA} \and John Jasper\thanks{Department of Mathematics and Statistics, Air Force Institute of Technology, Wright-Patterson AFB, Ohio, USA} \and Dustin G.\ Mixon\thanks{Department of Mathematics, The Ohio State University, Columbus, Ohio, USA} \thanks{Translational Data Analytics Institute, The Ohio State University, Columbus, Ohio, USA}}

\date{}

\begin{document}
\maketitle

\begin{abstract}
We introduce a new infinite family of $d\times 2d$ equiangular tight frames.
Many matrices in this family consist of two $d\times d$ circulant blocks.
We conjecture that such equiangular tight frames exist for every $d$.
We show that our conjecture holds for $d\leq 165$ by a computer-assisted application of a Newton--Kantorovich theorem.
In addition, we supply numerical constructions that corroborate our conjecture for $d\leq 1500$.
\end{abstract}

\section{Introduction}

Consider the problem of arranging a given number of points in complex projective space so as to maximize the minimum pairwise distance.
After identifying each point with a rank-$1$ orthogonal projection matrix, one might expect the optimal arrangement to form the vertices of a regular simplex in matrix space.
In fact, if the centroid of these pairwise equidistant vertices coincides with the centroid of projective space, namely, the unit-trace multiple of the identity matrix, then the points indeed form an optimal projective code known as an \textbf{equiangular tight frame (ETF)}; see~\cite{Welch:74,ConwayHS:96,StrohmerH:03,FickusM:15}.
It is convenient to represent a point in complex projective space $\mathbb{CP}^{d-1}$ with a unit vector in $\mathbb{C}^d$, in which case an $n$-point ETF can be thought of as a matrix $\Phi\in\mathbb{C}^{d\times n}$ with unit-norm columns such that the Gram matrix $\Phi^*\Phi$ has an off-diagonal of constant modulus and the so-called \textit{frame operator} $\Phi\Phi^*$ is a multiple of the identity matrix.

By virtue of their optimality as codes, ETFs find applications in multiple description coding~\cite{StrohmerH:03}, compressed sensing~\cite{BandeiraFMW:13}, digital fingerprinting~\cite{MixonQKF:13}, and quantum state tomography~\cite{RenesBSC:04}, which in turn has prompted a flurry of work to construct these objects; see~\cite{FickusM:15} for a survey.
As with many objects in combinatorial design, ETFs have a reputation for being miracles that only exist for very particular choices of $(d,n)$, although there are parameterized families of $(d,n)$ for which $d\times n$ ETFs are known to exist, and others for which they appear to exist.
Perhaps the most famous example is \textit{Zauner's conjecture}~\cite{Zauner:99}, which states that $d\times n$ ETFs (of a particular form) exist whenever $n=d^2$.
Fallon and Iverson~\cite{FallonI:23} recently posed an analogous conjecture for $n=2d$.

\begin{conjecture}[weak $d\times 2d$ conjecture; see~\cite{FallonI:23}]
For every $d$, there exists a $d\times 2d$ ETF.
\end{conjecture}

In this paper, which can be viewed as a sequel to~\cite{FallonI:23}, we make partial progress on (a couple of strengthenings of) this conjecture.

\subsection{Outline}

In Section~\ref{sec.background}, we review what is currently known about $d\times 2d$ ETFs.
Until the recent breakthrough of Fallon and Iverson~\cite{FallonI:23}, almost all of these ETFs were constructed in terms of conference matrices.
The new \textit{ETF doubling} construction in~\cite{FallonI:23} revealed that many more $d\times 2d$ ETFs exist, and this prompted Fallon and Iverson to pose the (weak) $d\times 2d$ conjecture. 
In Section~\ref{sec.doubling conference graphs}, we present a new ETF construction that instead doubles conference graphs.
Our approach produces infinitely many ETFs with new parameters, including the first known $17\times 34$ ETF, for example.
Next, Section~\ref{sec.multi generator} observes that many of our new ETFs are \textit{$2$-circulant}, meaning they consist of two $d\times d$ circulant matrices sitting side-by-side to form a single $d\times 2d$ matrix.
This leads us to pose a \textit{medium $d\times 2d$ conjecture} that for every $d$, there exists a $2$-circulant $d\times 2d$ ETF.
We report numerical constructions that corroborate this conjecture for $d\leq 1500$.
We also develop techniques to detect whether a given ETF is equivalent to one with this extra $2$-circulant structure.
We then apply these techniques in Section~\ref{sec.infinite families} to identify a few infinite families of $2$-circulant ETFs.
As an attempt to make further progress on this strengthened conjecture, we count the defining polynomials of $2$-circulant ETFs in Section~\ref{sec.continua}, and we observe that there are $\lceil \frac{3d}{2}\rceil$ \textit{more} real variables than real constraints.
This prompts a further strengthening of the $d\times 2d$ conjecture that for every $d\neq4$, there exists a $\lceil\frac{3d}{2}\rceil$-dimensional manifold of $2$-circulant ETFs.
(We exclude $d=4$ since it appears that the $4\times 8$ ETF is unique up to equivalence, and so the naive dimension count fails in this sporadic case.)
We then use a variant of the Newton--Kantorovich technology laid out by Cohn, Kumar, and Minton in~\cite{CohnKM:16} to verify this \textit{strong $d\times 2d$ conjecture} for every $d\leq 165$.
We conclude in Section~\ref{sec.discussion} with a summary of the current state of play for $d\leq150$ (as displayed in Table~\ref{table.first 150}) and a discussion of various opportunities for future work.

\subsection{Notation and nomenclature}

It is often convenient to work with matrices in which rows and columns are indexed by sets other than $[n]:=\{1,\ldots,n\}$.
For example, the rows and columns of an $n\times n$ circulant matrix are naturally indexed by the cyclic group $C_n$.
In such cases, given finite nonempty sets $A$ and $B$, we let $\mathbb{C}^{A\times B}$ denote the set of complex $|A|\times|B|$ matrices with rows indexed by $A$ and columns indexed by $B$.
Sometimes, we need to work with block matrices, in which case, for example, $(\mathbb{C}^{m\times n})^{p\times q}$ denotes the set of complex matrices that consist of a $p\times q$ array of $m\times n$ blocks.
Throughout, $I$ denotes the identity matrix and $J$ denotes the all-ones matrix, and the size of these matrices will be clear from context.
Similarly, $\mathbf{1}$ denotes the all-ones vector, although $\mathbf{1}_{\{P\}}$ denotes a scalar that equals $1$ when $P$ holds, and otherwise equals $0$.
Throughout, we denote $\mathrm{i}:=\sqrt{-1}$, which should not be confused with the index $i$ (mind change in font).

We frequently identify a sequence of vectors $\{\varphi_i\}_{i\in [n]}$ in $\mathbb{C}^d$ with the matrix $\Phi\in\mathbb{C}^{d\times n}$ whose $i$th column is $\varphi_i$.
The corresponding Gram matrix $\Phi^*\Phi$ consists of all $n^2$ inner products between $\{\varphi_i\}_{i\in [n]}$.
Our inner products are conjugate-linear in the left (not right!)\ entry so that $(\Phi^*\Phi)_{i,j}=\langle \varphi_i,\varphi_j\rangle$.
If $\Phi$ is a $d\times n$ ETF with $n>d$, then we may write $\Phi^*\Phi=I+\gamma S$, where $\gamma$ is the common modulus of the off-diagonal entries of $\Phi^*\Phi$.
Here, $S$ is known as the \textit{signature matrix} of $\Phi$.
In the case where $\Phi^*\Phi$ is real, $S$ is specified by a graph $G$ with vertex set $[n]$ in which adjacency between $i,j\in[n]$ is determined by value of $S_{i,j}$.
(We use the convention that $i\leftrightarrow j$ when $S_{i,j}=1$, but this choice will not matter.)
Here, $S$ is known as the \textit{Seidel adjacency matrix} of $G$.
Given a $d\times n$ ETF $\Phi$ with $n>d$ and signature matrix $S$, it holds that $-S$ is also a signature matrix for a class of $(n-d)\times n$ ETFs known as the \textit{Naimark complement} of $\Phi$.

Note that given an ETF, one may apply a unitary transformation on the left while phasing and reindexing the ETF vectors to obtain another ETF.
When two vector sequences are related in this way, we say they are \textit{switching equivalent}.
Thus, two ETFs are switching equivalent if and only if they represent the same multiset of lines.
In addition, switching equivalent ETFs are said to have switching equivalent Gram matrices and switching equivalent signature matrices.

\section{Background}
\label{sec.background}

In this section, we review what is currently known about $d\times 2d$ ETFs.

\subsection{Classical constructions}
\label{subsec.classic constructions}

The first constructions of $d\times 2d$ ETFs were based on \textbf{conference matrices}, that is, $n\times n$ matrices $C$ with $0s$ on the diagonal and $\pm1s$ on the off-diagonal such that $C^\top C=(n-1)I$.
Confrence matrices were introduced by Belevitch~\cite{Belevitch:50} in 1950 in the context of telephone communication.
Two conference matrices are said to be \textit{equivalent} if negating rows and columns of one produces the other.
Delsarte, Goethals, and Seidel~\cite{DelsarteGS:71} showed that every conference matrix with $n\equiv 2\bmod 4$ is equivalent to a symmetric conference matrix, while every conference matrix with $n\equiv 0\bmod 4$ is equivalent to a skew-symmetric conference matrix.
In either case, one may multiply by diagonal matrices of $\pm1$s to obtain a conference matrix of the form
\begin{equation}
\label{eq.sym of skew conf}
\left[\begin{array}{rl}
0&\mathbf{1}^\top\\
\mathbf{1}&S
\end{array}\right]
\qquad
\text{or}
\qquad
\left[\begin{array}{rl}
0&\mathbf{1}^\top\\
-\mathbf{1}&T
\end{array}\right],
\end{equation}
where $S$ is symmetric and $T$ is anti-symmetric.
In the former case, $S$ is the Seidel adjacency matrix of a \textbf{conference graph}, that is, a strongly regular graph on $v$ vertices such that every vertex has $k=\frac{v-1}{2}$ neighbors, every pair of adjacent vertices has $\lambda=\frac{v-5}{4}$ common neighbors, and every pair of distinct nonadjacent vertices has $\mu=\frac{v-1}{4}$ common neighbors.
In the latter case, one might say that $T$ is the Seidel adjacency matrix of a \textit{conference digraph}, though this nomenclature is not standard.

The construction of ETFs from conference matrices differs slightly between the symmetric and skew-symmetric cases; see Example~4.3 in~\cite{LemmensS:73} and Example~5.8 in~\cite{DelsarteGS:75}, respectively.
Given a symmetric conference matrix $C$, it holds that $C^2=(n-1)I$, and so the eigenvalues of $C$ reside in $\{\pm\sqrt{n-1}\}$.
Furthermore, since $C$ has trace zero, both eigenvalues have multiplicity $\frac{n}{2}$.
It follows that $I+\frac{1}{\sqrt{n-1}}C$ is the Gram matrix of a real $\frac{n}{2}\times n$ ETF.
Meanwhile, given a skew-symmetric conference matrix $C$, then $\mathrm{i}C$ is self-adjoint with $(\mathrm{i}C)^2=(n-1)I$, and so a similar argument gives that $I+\frac{\mathrm{i}}{\sqrt{n-1}}C$ is the Gram matrix of a complex $\frac{n}{2}\times n$ ETF.
In the sequel, we use $F_v$ to label any conference (di)graph on $v$ vertices, and we label any corresponding (anti-)symmetric conference matrix (along with the corresponding ETFs) by $F_v+1$.

\begin{example}
\label{ex.paley conference matrices}
Given a prime power $q\equiv 1\bmod 4$, the corresponding \textbf{Paley graph} has vertex set $\mathbb{F}_q$ with adjacency $\alpha\leftrightarrow\beta$ precisely when the difference $\alpha-\beta$ is a nonzero quadratic residue.
Note that the condition $q\equiv 1\bmod 4$ ensures that $-1$ is a quadratic residue so that the Paley graph is not directed.
It turns out that the Paley graph is a conference graph.
In the case where $q\equiv 3\bmod 4$, it holds that $-1$ is \textit{not} a quadratic residue, and so we may define the \textit{Paley digraph} by taking $\alpha\rightarrow\beta$ precisely when $\alpha-\beta$ is a nonzero quadratic residue.
It turns out that the Paley digraph is a conference digraph.
(The remaining case where $q$ is an even prime power is not interesting here since \textit{every} member of $\mathbb{F}_q^\times$ is a quadratic residue.)
Throughout, we let $G_q$ denote the Paley (di)graph of order $q$, and we label the resulting \textit{Paley conference matrices} (along with their negatives and the corresponding \textit{Paley ETFs}) by $G_q+1$.
In his PhD thesis, Zauner~\cite{Zauner:99} constructed an explicit $\frac{q+1}{2}\times (q+1)$ ETF for each odd prime power $q$ in terms of additive characters of $\mathbb{F}_q$, and then Et-Taoui~\cite{EtTaoui:16} later identified these ETFs as Paley ETFs by computing the corresponding conference matrices.
\end{example}

\begin{example}
\label{ex.symplectic form construction}
Given an odd prime power $q$, a $2$-dimensional vector space $V$ over $\mathbb{F}_q$ with symplectic form $[\cdot,\cdot]\colon V\times V\to\mathbb{F}_q$, and a choice of nonzero 
vectors $t_1,\ldots,t_{q+1}\in V$ that span distinct lines in $V$, consider the matrix $C$ defined by
\[
C_{i,j}
=\chi([t_i,t_j]),
\]
where $\chi\colon\mathbb{F}_q\to\mathbb{C}$ denotes the Legendre symbol.
Then $C$ is a conference matrix that is symmetric for $q\equiv 1\bmod 4$ and skew-symmetric for $q\equiv 3\bmod 4$.
See~\cite{Cameron:91} for a related construction.
\end{example}

Considering their use of $\mathbb{F}_q$, one might expect the conference matrices in Example~\ref{ex.symplectic form construction} to be equivalent in some sense to the Paley conference matrices in Example~\ref{ex.paley conference matrices}.
This is made explicit in the following folk result, which we prove in the appendix.

\begin{proposition}
\label{prop.paley conference matrices are the equivalent}
Take any odd prime power $q$.
\begin{itemize}
\item[(a)]
There exists a symplectic form on $\mathbb{F}_q^2$ and a choice of vectors $t_1,\ldots,t_{q+1}\in\mathbb{F}_q^2$ such that the conference matrix described in Example~\ref{ex.symplectic form construction} is precisely the Paley conference matrix of order $q+1$.
\item[(b)]
The conference matrices of order $q+1$ described in Example~\ref{ex.symplectic form construction} are all switching equivalent.
\end{itemize}
\end{proposition}

We conclude this subsection by commenting on how highly symmetric the ETFs arising from conference matrices tend to be.
Here, it is more appropriate to consider the \textit{lines} spanned by the ETF vectors rather than the vectors themselves.
First, Taylor~\cite{Taylor:92} showed that the Paley ETFs with $q\equiv 1\bmod 4$ are \textit{doubly transitive} in the sense that every ordered pair of distinct ETF lines can be mapped to any other ordered pair of distinct ETF lines by way of a unitary transformation that simultaneously permutes all of the ETF lines.
Later, Iverson and Mixon~\cite{IversonM:24} showed that the Paley ETFs with $q\equiv 3\bmod 4$ are also doubly transitive.
Finally, Et-Taoui~\cite{EtTaoui:00} showed that for any ETF arising from a skew-symmetric conference matrix, every unordered triple of ETF lines can be unitarily transformed to any other unordered triple of ETF lines, though not necessarily in a way that permutes the other lines.
See Theorem~3.10 in~\cite{HoffmanS:12} for a related result.

\subsection{Other ETFs of Paley size}
\label{subsec.other etfs of paley size}

In the previous subsection, every ETF signature matrix arising from a conference matrix has an off-diagonal consisting of fourth roots of unity.
To find new ETFs, one might hunt for additional ETF signature matrices with fourth roots of unity.
After modding out by the symmetries in the problem, Duncan, Hoffman, and Solazzo~\cite{DuncanHS:10} computationally discovered new $d\times 2d$ ETFs of this form with $d\in\{3, 5, 6, 7, 9\}$.
Later, Et-Taoui~\cite{EtTaoui:16} discovered a continuum of $d\times 2d$ ETFs that unified each of these ETFs (save the $d=6$ case) with the Paley ETFs.
In particular, every Paley conference matrix with $q=2d-1\equiv 1\bmod 4$ can be rearranged to take the form
\[
\left[\begin{array}{cccc}
0\phantom{^\top} & \phantom{-}1\phantom{^\top} & \phantom{-}\mathbf{1}^\top & \phantom{-}\mathbf{1}^\top \\
1\phantom{^\top} & \phantom{-}0\phantom{^\top} & -\mathbf{1}^\top & \phantom{-}\mathbf{1}^\top \\
\mathbf{1}\phantom{^\top} & -\mathbf{1}\phantom{^\top} & \phantom{-}A\phantom{^\top} & \phantom{-}B\phantom{^\top} \\
\mathbf{1}\phantom{^\top} & \phantom{-}\mathbf{1}\phantom{^\top} & \phantom{-}B^\top & -A\phantom{^\top}
\end{array}\right],
\]
and multiplying $B$ and $B^\top$ by \textit{any} unimodular complex scalar and its complex conjugate (respectively) results in a complex ETF signature matrix.

\subsection{Recent breakthrough using doubling}

Until recently, almost every known complex $d\times 2d$ ETF could be constructed (up to switching) from a real $2d\times 2d$ conference matrix using the techniques discussed in the previous two subsections.
(The only exceptional case was the complex $6\times 12$ ETF discovered by Duncan, Hoffman, and Solazzo~\cite{DuncanHS:10}.)
This state of affairs changed dramatically with a recent construction due to Fallon and Iverson~\cite{FallonI:23}.
To motivate this construction, first observe that skew-symmetric conference matrices are closed under the following operation:
\[
C
\mapsto
\left[\begin{array}{cc}
C & C+I\\
C-I & -C
\end{array}\right].
\]
This \textit{doubling} operation was first stated by J.\ Wallis~\cite{Wallis:71} in 1971 in the context of skew Hadamard matrices.
One may express doubling in terms of the corresponding ETF signature matrices as
\begin{equation}
\label{eq.first signature double}
S
\mapsto
\left[\begin{array}{cc}
S & S+\mathrm{i}I\\
S-\mathrm{i}I & -S
\end{array}\right].
\end{equation}
This operation is generalized in the following result.

\begin{proposition}[Theorem~6 in~\cite{FallonI:23}]
\label{prop.etf doubling}
Let $S$ denote the signature matrix of a complex $d\times n$ ETF, where $n=2d+k$ for some $k\in\{-1,0,1\}$.
Pick $\varepsilon\in\{\pm1\}$ and take
\[
c:=(n-2d) \sqrt{\frac{n-1}{d ( n-d ) }},
\qquad
\beta:=-c+\varepsilon \mathrm{i}\sqrt{1-c^2}.
\]
Then
\begin{equation}
\label{eq.doubled ETF signature matrix}
\left[\begin{array}{cc}
S & S+\beta I\\
S+\overline\beta I & -S
\end{array}\right]
\end{equation}
is the signature matrix of a complex $n\times 2n$ ETF.
\end{proposition}

\begin{example}
\label{ex.doubling the paley ETF}
If $S=\mathrm{i}C$ for some skew-symmetric conference matrix $C$ and $\varepsilon=1$, then the signature matrix~\eqref{eq.doubled ETF signature matrix} is identical to the result of~\eqref{eq.first signature double}.
By Proposition~\ref{prop.etf doubling}, we may also apply \eqref{eq.first signature double} to a \textit{real symmetric} conference matrix and obtain a \textit{complex} ETF signature matrix.
In the case where $S$ is the (real or complex) signature matrix of a Paley ETF $G_q+1$, we label this doubled signature matrix (and the corresponding ETFs) by $2\cdot (G_q+1)$.
\end{example}

\begin{example}
\label{ex.doubling the renes-strohmer ETF}
One may also apply Proposition~\ref{prop.etf doubling} to $d\times n$ ETFs with $n=2d\pm1$.
The best known example of such ETFs is the Renes--Strohmer ETF~\cite{Renes:07,Strohmer:08}.
Here, we take any $2d\times 2d$ skew-symmetric conference matrix $C$ of the form \eqref{eq.sym of skew conf}.
Since $T$ is skew-symmetric with trace zero, writing out $C^\top C=I$ in terms of $T$ reveals that $\mathrm{i}T$ is self-adjoint with three eigenspaces: the kernel spanned by $\mathbf{1}$, and two other eigenspaces of dimension $d-1$ with eigenvalues $\pm\sqrt{2d-1}$.
It follows that
\[
\mathrm{i}T+\frac{1}{\sqrt{2d-1}}\cdot J+\sqrt{2d-1}\cdot I
\]
is a multiple of the Gram matrix of a $d\times (2d-1)$ ETF, while the Naimark complement is a $d'\times(2d'+1)$ ETF with $d'=d-1$.
Overall, if there exists a conference digraph on $v$ vertices, then applying Proposition~\ref{prop.etf doubling} to the resulting Renes--Strohmer ETF produces a $v\times 2v$ ETF.
We label any such doubled ETF by $2\cdot F_{v}$ with $v\equiv 3\bmod 4$.
In the case where $C$ is a Paley conference matrix, meaning $T$ is the Seidel adjacency matrix of a Paley digraph, we label this doubled ETF by $2\cdot G_q$ with $q\equiv 3\bmod 4$.
\end{example}

\section{Doubling conference graphs}
\label{sec.doubling conference graphs}

In this section, we present a new doubling construction of $d\times 2d$ ETFs.

\begin{theorem}
\label{thm.doubling signature}
If there exists a conference graph on $v$ vertices, then there exists a complex $v\times 2v$ ETF.
Explicitly, let $A\in\mathbb{R}^{v\times v}$ denote the $\{0,1\}$-adjacency matrix of the conference graph, and let $B=J-I-A$ denote the adjacency matrix of the complementary graph.
Pick $\varepsilon\in\{\pm1\}$ and define
\[
x:=\frac{-1+\sqrt{2v-1}}{v-1},
\qquad
y:=\sqrt{1-x^2},
\qquad
\beta:=\varepsilon x+\mathrm{i}y.
\]
Then
\begin{equation}
\label{eq.doubling conference graph to signature matrix}
\left[\begin{array}{cc}
A-B & \varepsilon I+\beta A+\overline{\beta}B\\
\varepsilon I+\overline{\beta} A+\beta B & B-A
\end{array}\right]
\end{equation}
is the signature matrix of a complex $v\times2v$ ETF.
\end{theorem}

Note that $x\in[0,1]$ as a consequence of $v^2-2v+1=(v-1)^2\geq0$, i.e., $2v-1\leq v^2$.
It follows that $y\in[0,1]$ and $|\beta|=1$.
Also, the Naimark complementary signature matrix is obtained by swapping $A\leftrightarrow B$ and $\varepsilon\leftrightarrow-\varepsilon$.
We label the ETFs constructed by doubling a conference graph $F_v$ by $2\cdot F_v$ with $v\equiv 1\bmod 4$, extending our notation for the doubled Renes--Strohmer ETF in Example~\ref{ex.doubling the renes-strohmer ETF}.

\begin{proof}[Proof of Theorem~\ref{thm.doubling signature}]
Let $S$ denote the signature matrix \eqref{eq.doubling conference graph to signature matrix}.
It suffices to verify $S^2=(2v-1)I$.
The blocks of $S$ reside in the complex commutative $*$-algebra spanned by $\{I,A,B\}$, so we start by recalling the structure constants (also known as \textit{intersection numbers}) of this adjacency algebra~\cite{Godsil:10}:
\begin{align*}
A^2&=\frac{1}{4}\Big( (2v-2)I+(v-5)A+(v-1)B \Big),\\
AB=BA&=\frac{1}{4}\Big( (v-1)A+(v-1)B \Big),\\
B^2&=\frac{1}{4}\Big( (2v-2)I+(v-1)A+(v-5)B\Big).
\end{align*}
(These can be verified using the fact that a conference graph on $v$ vertices is strongly regular with parameters $k=\frac{v-1}{2}$, $\lambda=\frac{v-5}{4}$, and $\mu=\frac{v-1}{4}$.)
We will use these relations along with $\beta\overline{\beta}=1$ to compute each block of $S^2\in(\mathbb{C}^{v\times v})^{2\times2}$.
First, a lengthly but straighforward calculation gives
\begin{align*}
[S^2]_{11}
&=(A-B)(A-B)+(\varepsilon I+\beta A+\overline{\beta}B)(\varepsilon I+\overline{\beta}A+\beta B)\\
&=(2v-1)I+\frac{1}{2}\Big(v-5+(v-1)\operatorname{Re}[\beta^2]+4\varepsilon\operatorname{Re}[\beta]\Big)(A+B).
\end{align*}
We simplify this further using the facts that $\operatorname{Re}[\beta]=\varepsilon x$ and
\[
\operatorname{Re}[\beta^2]
=\operatorname{Re}[(\varepsilon x+iy)^2]
=\operatorname{Re}[x^2-y^2+2\varepsilon xyi]
=x^2-y^2
=2x^2-1.
\]
In particular, the coefficient of $A+B$ above simplifies to
\[
\frac{1}{2}\Big(v-5+(v-1)\operatorname{Re}[\beta^2]+4\varepsilon\operatorname{Re}[\beta]\Big)
=(v-1)x^2+2x-2,
\]
which vanishes by our choice of $x$.
Thus, $[S^2]_{11}=(2v-1)I$.
Next,
\[
[S^2]_{22}
=(\varepsilon I+\overline{\beta}A+\beta B)(\varepsilon I+\beta A+\overline{\beta} B)+(B-A)(B-A)
=[S^2]_{11}
=(2v-1)I.
\]
Finally,
\[
[S^2]_{12}
=(A-B)(\varepsilon I+\beta A+\overline{\beta}B)+(\varepsilon I+\beta A+\overline{\beta}B)(B-A)
=0,
\]
and $[S^2]_{21}=([(S^2)^*]_{12})^*=([S^2]_{12})^*=0$.
Overall, $S^2=(2v-1)I$, as desired.
\end{proof}

\begin{example}
\label{eq.double paley graph}
Applying Theorem~\ref{thm.doubling signature} to the Paley graph produces a $q\times 2q$ ETF.
We label this ETF by $2\cdot G_q$ with $q\equiv1\bmod 4$, thereby extending our notation from Example~\ref{ex.doubling the renes-strohmer ETF}.
This already gives an ETF with new parameters when $q=17$.
\end{example}

\begin{corollary}
If there exists a real $d\times 2d$ ETF, then there exists a complex $D\times 2D$ ETF with $D=2d-1$.
\end{corollary}

\begin{proof}
The signature matrix of a real $d\times 2d$ ETF is switching equivalent to the Seidel adjacency matrix of a conference graph union an isolated vertex, meaning there exists a conference graph on $2d-1$ vertices.
The result then follows from Theorem~\ref{thm.doubling signature}.
\end{proof}

The following corollary summarizes both cases of the $2\cdot F_v$ construction of ETFs (see Example~\ref{ex.doubling the renes-strohmer ETF} and the paragraph after Theorem~\ref{thm.doubling signature}) in terms of conference matrices of order $n=v+1$.

\begin{corollary}
If there exists a conference matrix of order $n$, then there exists a complex $d\times 2d$ ETF with $d=n-1$.
\end{corollary}

\begin{proof}
Delsarte, Goethals, and Seidel~\cite{DelsarteGS:71} showed that a conference matrix of order $n$ exists only if $n$ is even, and furthermore, any such matrix is equivalent to either a symmetric or a skew-symmetric matrix depending on whether $n \equiv 2$ or $0 \bmod 4$, respectively. 
In the symmetric case, after normalizing the first row and column, the core of a symmetric conference matrix is the Seidel adjacency matrix of a conference graph on $n-1$ vertices, and so Theorem~\ref{thm.doubling signature} delivers a complex $d\times 2d$ ETF with $d = n-1$. 
In the skew-symmetric case, one may apply Proposition~\ref{prop.etf doubling} to the resulting Renes--Strohmer ETF as in Example~\ref{ex.doubling the renes-strohmer ETF} to obtain a complex $d\times2d$ ETF with $d = n-1$.
\end{proof}

Next, we describe a complex $d\times 2d$ ETF whose signature matrix is given in Theorem~\ref{thm.doubling signature}.
We start with a helpful lemma.

\begin{lemma}
\label{lem.low rank extension}
Suppose $A$ is an invertible $k\times k$ matrix and $\operatorname{rank}\left(\left[\begin{smallmatrix}A&B\\C&D\end{smallmatrix}\right]\right)
=k$.
Then $D = C A^{-1} B$.
\end{lemma}

\begin{proof}
By assumption, the image of $\left[\begin{smallmatrix}B\\D\end{smallmatrix}\right]$ is contained in the image of $\left[\begin{smallmatrix}A\\C\end{smallmatrix}\right]$.
That is, for every appropriately sized vector $x$, there exists a vector $y$ such that $Bx = Ay$ and $Dx = Cy$.
Since $A$ is invertible, it necessarily holds that $y = A^{-1} B x$.
Thus, $D x = C A^{-1} B x$ for every $x$, and so $D = C A^{-1} B$.
\end{proof}

\begin{theorem}
\label{thm.short fat repn of doubling}
Let $A\in\mathbb{R}^{v\times v}$ denote the $\{0,1\}$-adjacency matrix of a conference graph on $v$ vertices, and let $B=J-I-A$ denote the adjacency matrix of the complementary graph.
Then there exist $a,b,c,d,e,f\in\mathbb{C}$ such that
\[
\left[\begin{array}{c|c}
aI+bA+cB & dI+eA+fB
\end{array}\right]
\]
is a $v\times 2v$ complex ETF.
Explicitly, given
\[
\gamma:=\frac{1}{\sqrt{2v-1}},
\qquad
k:=\frac{v-1}{2},
\qquad
r:=\frac{-1+\sqrt{v}}{2},
\qquad
s:=\frac{-1-\sqrt{v}}{2},
\]
as well as $\varepsilon$ and $\beta$ from Theorem~\ref{thm.doubling signature}, we may take
\begin{alignat*}{2}
&
a
=\tfrac{1}{v}\Big(1+k\sqrt{1+\gamma(r-s)}+k\sqrt{1-\gamma(r-s)}\Big),
&\quad&
d
=\tfrac{\gamma}{v}\Big(
\varepsilon+2k\operatorname{Re}\beta+k\tfrac{\varepsilon+\beta r+\overline{\beta}s}{\sqrt{1+\gamma(r-s)}}+k\tfrac{\varepsilon+\beta s+\overline{\beta}r}{\sqrt{1-\gamma(r-s)}}\Big),\\
&
b
=\tfrac{1}{v}\Big(1+r\sqrt{1+\gamma(r-s)}+s\sqrt{1-\gamma(r-s)}\Big),
&&
e
=\tfrac{\gamma}{v}\Big(
\varepsilon+2k\operatorname{Re}\beta+r\tfrac{\varepsilon+\beta r+\overline{\beta}s}{\sqrt{1+\gamma(r-s)}}+s\tfrac{\varepsilon+\beta s+\overline{\beta}r}{\sqrt{1-\gamma(r-s)}}\Big),\\
&
c
=\tfrac{1}{v}\Big(1+s\sqrt{1+\gamma(r-s)}+r\sqrt{1-\gamma(r-s)}\Big),
&&
f
=\tfrac{\gamma}{v}\Big(
\varepsilon+2k\operatorname{Re}\beta+s\tfrac{\varepsilon+\beta r+\overline{\beta}s}{\sqrt{1+\gamma(r-s)}}+r\tfrac{\varepsilon+\beta s+\overline{\beta}r}{\sqrt{1-\gamma(r-s)}}\Big).
\end{alignat*}
\end{theorem}

\begin{proof}
Given the signature matrix $S\in(\mathbb{C}^{v\times v})^{2\times 2}$ from Theorem~\ref{thm.doubling signature}, let $G:=I+\gamma S\in(\mathbb{C}^{v\times v})^{2\times 2}$ denote the corresponding Gram matrix.
We will prove four things about the blocks $G_{11},G_{12}\in\mathbb{C}^{v\times v}$, which together imply the result:
\begin{itemize}
\item[(i)]
$G_{11}$ is positive definite.
\item[(ii)]
$\left[\begin{array}{cc}G_{11}^{1/2}&G_{11}^{-1/2}G_{12}
\end{array}\right]$ is a $v\times 2v$ complex ETF with Gram matrix $G$.
\item[(iii)]
$G_{11}^{1/2}$ and $G_{11}^{-1/2}G_{12}$ both reside in the complex span of $\{I,A,B\}$.
\item[(iv)]
The coefficients of $G_{11}^{1/2}$ and $G_{11}^{-1/2}G_{12}$ in $\{I,A,B\}$ are given by $a,b,c,d,e,f$ above, respectively.
\end{itemize}
First, we use Lemma~\ref{lem.low rank extension} to verify (ii) under the assumption that (i) holds:
\[
\left[\begin{array}{cc}G_{11}^{1/2}&G_{11}^{-1/2}G_{12}
\end{array}\right]^*
\left[\begin{array}{cc}G_{11}^{1/2}&G_{11}^{-1/2}G_{12}
\end{array}\right]
=\left[\begin{array}{cc}
G_{11} & G_{12} \\
G_{12}^* & G_{12}^* G_{11}^{-1} G_{12}
\end{array}\right]
=\left[\begin{array}{cc}
G_{11} & G_{12} \\
G_{21} & G_{21} G_{11}^{-1} G_{12}
\end{array}\right]
=G.
\]
For the remaining items, we recall that $I$, $A$, and $B$ span a complex commutative $*$-algebra $\mathcal{A}$.
The primitive idempotents $L,P,Q\in\mathcal{A}$ form an alternate orthogonal basis for $\mathcal{A}$ such that
\[
I = L+P+Q,
\qquad
A = kL+rP+sQ,
\qquad
B = kL+sP+rQ.
\]
(These eigenvalues can be derived from the spectral theory of conference graphs~\cite{BrouwerH:11}.)
It follows that the eigenvalues of $G_{11}=I+\gamma(A-B)=G_{11}^*$ are $1$ and $1\pm\gamma(r-s)$, all of which are strictly positive, meaning (i) holds.
For (iii), one may compute $G_{11}^{1/2}$ and $G_{11}^{-1/2}$ by taking (reciprocals of) the square roots of the coefficients of $G_{11}$ in $\{L,P,Q\}$.
It follows that $G_{11}^{1/2},G_{11}^{-1/2}\in\mathcal{A}$ and furthermore, $G_{11}^{-1/2}G_{12}\in\mathcal{A}$.
Finally, (iv) follows by actually performing these (straightforward but tedious) computations.
\end{proof}

\section{Multi-generator harmonic collections of vectors}
\label{sec.multi generator}

In this section, we study the symmetries exhibited by many of the ETFs constructed in the previous section, and then we use these symmetries to refine the $d\times 2d$ conjecture.
We start by considering some low-dimensional examples.

\begin{example}
\label{ex.paley 5}
Applying Theorem~\ref{thm.short fat repn of doubling} with $\varepsilon=1$ to the cyclic graph on $5$ vertices produces the ETF
\[
\left[
\begin{array}{ccccc|ccccc}
a & b & c & c & b & d & e & f & f & e \\
b & a & b & c & c & e & d & e & f & f \\
c & b & a & b & c & f & e & d & e & f \\
c & c & b & a & b & f & f & e & d & e \\
b & c & c & b & a & e & f & f & e & d
\end{array}
\right],
\qquad
\quad
\begin{array}{ll}
a = \tfrac{1}{5}+2\sqrt{\tfrac{2}{15}},
\qquad
&d = \tfrac{1}{5}-\operatorname{exp}(\tfrac{2\pi \mathrm{i}} {3})\sqrt{\tfrac{2}{15}},\\[2ex]
b = \tfrac{1}{5},
\qquad
&e = \tfrac{6-\sqrt{-15(13+3\mathrm{i}\sqrt{3})}}{30},\\[2ex]
c = \tfrac{1}{5}-\sqrt{\tfrac{2}{15}},
\qquad
&f = \tfrac{1}{5}-\tfrac{\mathrm{i}}{\sqrt{10}}.
\end{array}
\]
Notably, this ETF consists of two circulant matrices.
This occurs more generally when applying Theorem~\ref{thm.short fat repn of doubling} to a circulant conference graph, e.g., a Paley graph with a prime number of vertices.
\end{example}

\begin{example}
In the special case of the cyclic graph on $5$ vertices, one may replace the off-diagonal blocks in equation~\eqref{eq.doubling conference graph to signature matrix} to obtain the signature matrix of a real $5\times 10$ ETF.
The proof of Theorem~\ref{thm.short fat repn of doubling} then gives the ETF
\[
\left[
\begin{array}{ccccc|ccccc}
a & b & c & c & b & d & e & f & f & e \\
b & a & b & c & c & e & d & e & f & f \\
c & b & a & b & c & f & e & d & e & f \\
c & c & b & a & b & f & f & e & d & e \\
b & c & c & b & a & e & f & f & e & d
\end{array}
\right],
\qquad
\quad
\begin{array}{ll}
a = \tfrac{1}{5}+2\sqrt{\tfrac{2}{15}},
\qquad
&d = -\tfrac{1}{5}+2\sqrt{\tfrac{2}{15}},\\[2ex]
b = \tfrac{1}{5},
\qquad
&e = -\tfrac{1}{5}-\sqrt{\tfrac{2}{15}},\\[2ex]
c = \tfrac{1}{5}-\sqrt{\tfrac{2}{15}},
\qquad
&f = -\tfrac{1}{5}.
\end{array}
\]
As before, this ETF consists of two circulant matrices, but they are both real in this case.
\end{example}

The above examples suggest the following definition.

\begin{definition}
A $d\times td$ matrix in $(\mathbb{C}^{d\times d})^{1\times t}$ is \textbf{$t$-circulant} if each $d\times d$ block is a circulant matrix.
\end{definition}

As the following examples illustrate, many important ETFs are $t$-circulant in the appropriate basis.

\begin{example}
\label{ex.3x6}
Recall that a real $3\times 6$ matrix with unit-norm columns is an ETF if and only if its columns and their negatives form the vertices of a regular icosahedron~\cite{Waldron:09}.
The regular icosahedron exhibits a symmetry by rotating $\frac{2\pi}{3}$ radians about an axis through the center of one of its triangular faces.
(See Figure~\ref{figure.icosahedron} for an illustration.)
If we regard one direction of this axis as the north pole, then the $12$ vertices reside in $4$ lines of latitude.
If we orient the north pole to point in the $(1,1,1)$ direction, then in the appropriate coordinate system, the vertices in the first and third lines of latitude are given by the cyclic translations of
\[
\tfrac{1}{\sqrt{1+\phi^2}}\cdot(0,1,\phi)
\qquad
\text{and}
\qquad
\tfrac{1}{\sqrt{1+\phi^2}}\cdot(0,1,-\phi),
\]
respectively, where $\phi=\frac{1+\sqrt{5}}{2}$ is the golden ratio.
This $2$-circulant representation of the real $3\times 6$ ETF stems from the doubling construction found in Example~7 of~\cite{FallonI:23}, and we will generalize this correspondence in the next section.
\end{example}

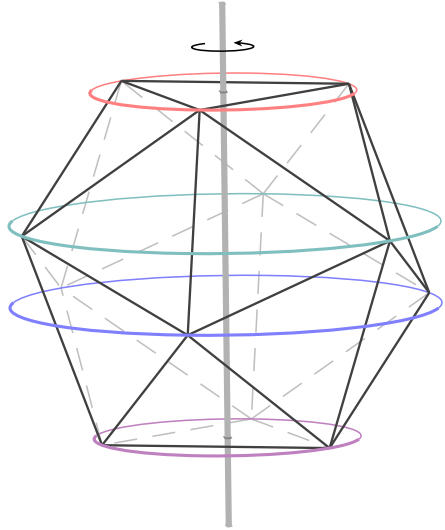
\begin{figure}
\begin{center}
\tdplotsetmaincoords{0}{0}
\resizebox{!}{200 pt}{
\begin{tikzpicture}[tdplot_main_coords]
	\tdplotsetrotatedcoords{38}{-13}{-38}
	\begin{scope}[tdplot_rotated_coords]

	%vertex coordinates
	\coordinate (0) at (0,0,0);
	\coordinate (1) at (0,0.187592474085080,-0.982246946376846);
	\coordinate (2) at (-0.850650808352040,0.187592474085080,0.491123473188423);
	\coordinate (3) at (0.850650808352040,0.187592474085080,0.491123473188423);
	\coordinate (4) at (0, -0.794654472291766,-0.607061998206686);
	\coordinate (5) at (-0.525731112119134,-0.794654472291766,0.303530999103343);
	\coordinate (6) at (0.525731112119134,-0.794654472291766,0.303530999103343);
	\coordinate (7) at (0,-0.187592474085080,0.982246946376846);
	\coordinate (8) at (0.850650808352040,-0.187592474085080,-0.491123473188423);
	\coordinate (9) at (-0.850650808352040,-0.187592474085080,-0.491123473188423);
	\coordinate (10) at (0, 0.794654472291766,0.607061998206686);
	\coordinate (11) at (0.525731112119134,0.794654472291766,-0.303530999103343);
	\coordinate (12) at (-0.525731112119134,0.794654472291766,-0.303530999103343);
        \coordinate (N) at (0,1.2,0);
        \coordinate (S) at (0,-1.2,0);

	%edges in the background
	\draw [densely dashed, very thin, lightgray] (1) -- (4);
	\draw [densely dashed, very thin, lightgray] (1) -- (8);
	\draw [densely dashed, very thin, lightgray] (1) -- (9);
	\draw [densely dashed, very thin, lightgray] (1) -- (11);
	\draw [densely dashed, very thin, lightgray] (1) -- (12);
	\draw [densely dashed, very thin, lightgray] (2) -- (9);
	\draw [densely dashed, very thin, lightgray] (4) -- (5);
	\draw [densely dashed, very thin, lightgray] (4) -- (6);
	\draw [densely dashed, very thin, lightgray] (4) -- (8);
	\draw [densely dashed, very thin, lightgray] (4) -- (9);
	\draw [densely dashed, very thin, lightgray] (5) -- (9);
	\draw [densely dashed, very thin, lightgray] (9) -- (12);

        %circles
        
        \draw [thick, black!30] (S) -- (0,-0.794654472291766,0);
        \begin{scope}[canvas is xz plane at y=1.2,transform shape]
            \draw [black!30] (0,0) circle [radius=0.007];
        \end{scope} 
        \begin{scope}[canvas is xz plane at y=0.794654472291766,transform shape]
            \draw [very thin, red!50] (-0.6070619982,0) arc (180:360:0.6070619982);
        \end{scope} 
        \begin{scope}[canvas is xz plane at y=0.187592474085080,transform shape]
            \draw [very thin, teal!50] (-0.98224694637,0) arc (180:360:0.98224694637);
        \end{scope}
        \begin{scope}[canvas is xz plane at y=-0.187592474085080,transform shape]
            \draw [very thin, blue!50] (-0.98224694637,0) arc (180:360:0.98224694637);
        \end{scope}
        \begin{scope}[canvas is xz plane at y=-0.794654472291766,transform shape]
            \draw [black!50] (0,0) circle [radius=0.01];
            \draw [very thin, violet!50] (-0.6070619982,0) arc (180:360:0.6070619982);
        \end{scope} 
        \begin{scope}[canvas is xz plane at y=-1.2,transform shape]
            \draw [black!30] (0,0) circle [radius=0.007];
        \end{scope}

        \draw [thick, black!30] (0,-0.794654472291766,0) -- (0,0.794654472291766,0);
        \begin{scope}[canvas is xz plane at y=-0.792,transform shape]
        \end{scope} 
 
	%edges in the foreground
	\draw [darkgray, line width=0.3pt] (2) -- (5);
	\draw [darkgray, line width=0.3pt] (2) -- (7);
	\draw [darkgray, line width=0.3pt] (2) -- (10);
	\draw [darkgray, line width=0.3pt] (2) -- (12);
	\draw [darkgray, line width=0.3pt] (3) -- (6);
	\draw [darkgray, line width=0.3pt] (3) -- (7);
	\draw [darkgray, line width=0.3pt] (3) -- (8);
	\draw [darkgray, line width=0.3pt] (3) -- (10);
	\draw [darkgray, line width=0.3pt] (3) -- (11);
	\draw [darkgray, line width=0.3pt] (5) -- (6);
	\draw [darkgray, line width=0.3pt] (5) -- (7);
	\draw [darkgray, line width=0.3pt] (6) -- (7);
	\draw [darkgray, line width=0.3pt] (6) -- (8);
	\draw [darkgray, line width=0.3pt] (7) -- (10);
	\draw [darkgray, line width=0.3pt] (8) -- (11);
	\draw [darkgray, line width=0.3pt] (10) -- (11);
	\draw [darkgray, line width=0.3pt] (10) -- (12);
	\draw [darkgray, line width=0.3pt] (11) -- (12);

\begin{scope}[canvas is xz plane at y=0.794654472291766,transform shape]
            \draw [black!50] (0,0) circle [radius=0.01];
            \draw [red!50] (0.6070619982,0) arc (0:180:0.6070619982);
        \end{scope} 
        \begin{scope}[canvas is xz plane at y=0.187592474085080,transform shape]
            \draw [teal!50] (0.98224694637,0) arc (0:180:0.98224694637);
        \end{scope}
        \begin{scope}[canvas is xz plane at y=-0.187592474085080,transform shape]
            \draw [blue!50] (0.98224694637,0) arc (0:180:0.98224694637);
        \end{scope}
        \begin{scope}[canvas is xz plane at y=-0.794654472291766,transform shape]
            \draw [violet!50] (0.6070619982,0) arc (0:180:0.6070619982);
        \end{scope} 
        \draw [thick, black!30] (N) -- (0,0.794654472291766,0);
        \begin{scope}[canvas is xz plane at y=1,transform shape]
            \usetikzlibrary{arrows.meta}
            \draw [very thin, -{Stealth[scale=0.3]}] (-0.1,-0.1) arc (225:-80:0.141);
        \end{scope} 

	\end{scope}
\end{tikzpicture}
}
\end{center}
\caption{\label{figure.icosahedron} Illustration of Example~\ref{ex.3x6}. The vertices of a regular icosahedron partition into four lines of latitude. If we select coordinates so that the north pole points in the $(1,1,1)$ direction, then the vertices in the red and blue lines of latitude form a $2$-circulant representation of the real $3\times 6$ ETF.}
\end{figure}

\begin{example}
\label{ex.steiner}
Recall that a \textit{planar cyclic difference set} is a subset $D\subseteq C_v$ of size $k$ such that every nonzero member of $C_v$ can be expressed uniquely as a difference of members of $D$.
By counting, it necessarily holds that $v=m^2+m+1$ and $k=m+1$ for some positive integer $m$.
Such an object exists if (and conjecturally only if) $m$ is a prime power~\cite{Singer:38,Gordon:94}.
Now consider a (possibly complex) $(k+1)\times(k+1)$ Hadamard matrix $H$ with rows indexed by $D\cup\{\infty\}$ and columns indexed by $[k+1]$.
Then for each $i\in[k+1]$, the $i$th column of $H$ determines a vector $\varphi_i\in\mathbb{C}^{C_v}$ as follows:
\[
\varphi_i(g)
= \left\{\begin{array}{cl}
\frac{1}{\sqrt{k}}H_{g,i} & \text{if } g\in D\\
0 & \text{else.}
\end{array}\right.
\]
Recalling the general theory of \textit{Steiner ETFs}~\cite{FickusMT:12}, it follows that the cyclic translations of $\varphi_1,\ldots,\varphi_{k+1}$ form a $(k+1)$-circulant ETF.
For example, this construction produces a complex $3$-circulant $3\times 9$ ETF and a real $4$-circulant $7\times 28$ ETF.
\end{example}

\begin{example}
\label{ex.complex maximal}
The complex $3$-circulant $3\times 9$ ETF constructed in Example~\ref{ex.steiner} is extremal in the sense that a complex $d\times n$ ETF exists only if $n\leq d^2$; this is known as \textit{Gerzon's bound}~\cite{LemmensS:73-lines}.
In his PhD thesis, Zauner~\cite{Zauner:99} conjectured that for every $d$, there exists an ETF that saturates Gerzon's bound, and furthermore, such an ETF can be realized by taking all $d$ circular translations of each of $d$ modulations of a seed vector.
(Zauner also conjectured additional structure in the seed vector that we will not discuss here.)
Notably, such ETFs are necessarily $d$-circulant.
To date, Zauner's conjecture has been verified in scores of dimensions (see~\cite{ApplebyCFW:18}, for example), and numerical evidence supports the conjecture in hundreds of additional dimensions (see~\cite{ScottG:10}, for example).
The conjecture appears to follow from a strengthened version of the Stark conjectures from algebraic number theory (see~\cite{Kopp:21}, for example).
\end{example}

Returning to the primary subject of this paper, the prevalence of $2$-circulant ETFs that arise in Example~\ref{ex.paley 5} from doubling the Paley graph suggests a strengthening of the $d\times 2d$ conjecture.

\begin{conjecture}[medium $d\times 2d$ conjecture]
\label{conj.medium d by 2d}
For every $d$, there exists a $2$-circulant $d\times 2d$ ETF.
\end{conjecture}

As evidence in favor of Conjecture~\ref{conj.medium d by 2d}, the arXiv version of this paper includes an ancillary file that contains numerical constructions for $d\leq 1500$.

It is frequently convenient to work with the following generalization of $t$-circulant matrices.

\begin{definition}
Given a finite abelian group $\Gamma$, a unitary representation $\rho\colon\Gamma\to U(d)$, and generators $\psi_1,\ldots,\psi_t\in\mathbb{C}^d$, we say $\{\rho(g)\psi_i\}_{g\in\Gamma,i\in[t]}$ is \textbf{$t$-generator $\Gamma$-harmonic} with underlying representation~$\rho$.
\end{definition}

\begin{example}
The \textit{regular representation} of a finite abelian group $\Gamma$ is $\rho\colon\Gamma\to U(\mathbb{C}^\Gamma)$ given by $\rho(g)\varphi_h=\varphi_{h-g}$ for $\varphi\in\mathbb{C}^\Gamma$ and $g,h\in\Gamma$.
For instance, when $\Gamma=C_d$ is cyclic, the regular representation implements cyclic translations.
Since every circulant matrix can be obtained by applying such translations to its columns, every $t$-circulant matrix is $t$-generator $C_d$-harmonic.
\end{example}

Note that $t$-generator $C_d$-harmonic does not imply $t$-circulant unless the underlying representation $\rho$ is the regular representation.
For the reader who is familiar with existing frame-theoretic nomenclature, \textit{harmonic frames} are precisely the $1$-generator $\Gamma$-harmonic frames that happen to be tight.
In what follows, we characterize the Gram matrices of $t$-generator $\Gamma$-harmonic collections of vectors.
Below, $\hat\Gamma$ denotes the Pontryagin dual of a finite abelian group $\Gamma$; it consists of all group homomorphisms $\alpha\colon\Gamma\to\mathbb{T}$, and it is a group under pointwise multiplication.

\begin{theorem}
\label{thm.t-gen Gamma-harm detector}
Given an abelian group $\Gamma$ of order $m$ and $G\in(\mathbb{C}^{\Gamma\times \Gamma})^{t\times t}$, the following are equivalent:
\begin{itemize}
\item[(a)]
$G$ is the Gram matrix of a $t$-generator $\Gamma$-harmonic collection of vectors.
\item[(b)]
$G$ is positive semidefinite and each block $G_{i,j}\in\mathbb{C}^{\Gamma\times\Gamma}$ is $\Gamma$-stable, that is, 
\[
(G_{i,j})_{k+g,k+h}=(G_{i,j})_{g,h}
\quad
\text{for every} 
\quad
g,h,k\in \Gamma.
\]
\item[(c)]
Each block $G_{i,j}\in\mathbb{C}^{\Gamma\times \Gamma}$ is diagonalized by the Fourier transform matrix $F\in\mathbb{C}^{\hat\Gamma\times\Gamma}$ defined by
\[
F_{\alpha,g}
=\frac{1}{\sqrt{m}}\overline{\alpha(g)}
\qquad
(\alpha\in\hat\Gamma,\,g\in\Gamma)
\]
and each  $H_\alpha\in\mathbb{C}^{t\times t}$ defined by 
\[
(H_\alpha)_{i,j}
=(FG_{i,j}F^{-1})_{\alpha,\alpha}
\qquad
(i,j\in[t])
\]
is positive semidefinite.
\end{itemize}
If in addition to the above equivalent conditions, $G$ satisfies both 
\begin{equation}
\label{eq.reg repn detector}
G^2=tG
\qquad
\text{and}
\qquad
\sum_{i\in[t]} G_{i,i}=tI,
\end{equation}
then $G$ is the Gram matrix of a $t$-generator $\Gamma$-harmonic collection of vectors in $\mathbb{C}^m$, where the underlying representation is the regular representation of $\Gamma$.
In particular, if this holds and $\Gamma=C_m$, then $G$ is the Gram matrix of a $t$-circulant ETF.
\end{theorem}

\begin{proof}
For (a)$\Rightarrow$(b), suppose there exists a $t$-generator $\Gamma$-harmonic collection $\{\rho(g)\psi_j\}_{g\in\Gamma,j\in[t]}$ of vectors such that $(G_{i,j})_{g,h}=\langle \rho(g)\psi_i,\rho(h)\psi_j\rangle$.
Then $G$ is positive semidefinite as a consequence of being a Gram matrix.
Furthermore, since $\rho$ is a unitary representation, we have
\[
(G_{i,j})_{k+g,k+h}
=\langle \rho(k+g)\psi_i,\rho(k+h)\psi_j\rangle
=\langle \rho(k)\rho(g)\psi_i,\rho(k)\rho(h)\psi_j\rangle
=\langle \rho(g)\psi_i,\rho(h)\psi_j\rangle
=(G_{i,j})_{g,h},
\]
i.e., the $(i,j)$-block of $G$ is $\Gamma$-stable.

For (b)$\Rightarrow$(c), we recall that every $\Gamma$-stable matrix (e.g., each block of $G$) is diagonalized by $F$.
Thus, conjugating $G$ by $I\otimes F$ results in a $t\times t$ block matrix with diagonal blocks.
Conjugate by an appropriate permutation matrix to shuffle these rows and columns into a block diagonal matrix $H$ with blocks $\{H_\alpha\}_{\alpha\in\hat\Gamma}$.
Since $G$ was positive semidefinite by assumption, and $H$ is unitarily equivalent to $G$, it follows that each $H_\alpha$ is also positive semidefinite.

For (c)$\Rightarrow$(a), take any matrix $G$ satisfying (c).
For each $\alpha\in\hat\Gamma$ for which $H_{\overline\alpha}$ is nonzero, select any tuple $\{\psi^{(\alpha)}_i\}_{i\in[t]}$ of vectors in $\mathbb{C}^{\operatorname{rank}(H_{\overline\alpha})}$ with Gram matrix $\frac{1}{m}H_{\overline\alpha}$.
We direct sum these vectors to form $\psi_i:=\bigoplus_{\alpha\in\hat\Gamma,H_{\overline\alpha}\neq0}\psi^{(\alpha)}_i$ for each $i\in[t]$.
Let $P_\alpha$ denote the orthogonal projection onto the $\alpha$th direct summand (in particular, $P_\alpha=0$ if $H_{\overline\alpha}=0$), 
and define $\rho$ to act as $\alpha$ on the $\alpha$th direct summand:
\[
\rho(g)
=\sum_{\alpha\in\hat\Gamma}\alpha(g)P_\alpha.
\]
We claim that the $t$-generator $\Gamma$-harmonic collection $\{\rho(g)\psi_i\}_{g\in\Gamma,i\in[t]}$ of vectors has Gram matrix equal to $G$.
To see this, let $\tilde{G}$ denote the Gram matrix.
Since (a)$\Rightarrow$(c) by the above, the blocks of $\tilde{G}$ are diagonalized by $F$.
To obtain the diagonalization, we select an arbitrary character $\alpha\in\hat\Gamma$ and compute:
\begin{align*}
(F\tilde{G}_{i,j}F^{-1})_{\alpha,\alpha}
&=\sum_{g\in\Gamma}\sum_{h\in \Gamma}\frac{1}{\sqrt{m}}\overline{\alpha(g)}\cdot\langle \rho(g)\psi_i,\rho(h)\psi_j\rangle\cdot\frac{1}{\sqrt{m}}\alpha(h)\\
&=\frac{1}{m}\bigg\langle \sum_{g\in \Gamma}\alpha(g)\rho(g)\psi_i,\sum_{h\in \Gamma}\alpha(h)\rho(h)\psi_j\bigg\rangle\\
&=m\langle P_{\overline\alpha}\psi_i,P_{\overline\alpha}\psi_j\rangle
=m\langle \psi^{(\overline\alpha)}_i,\psi^{(\overline\alpha)}_j\rangle\mathbf{1}_{\{H_{\overline\alpha}\neq0\}}
=(H_{\alpha})_{i,j}.
\end{align*}
That is, for every $i,j\in[t]$, it holds that $F\tilde{G}_{i,j}F^{-1}$ and $FG_{i,j}F^{-1}$ are diagonal matrices with the same diagonal entries, i.e., $\tilde{G}=G$.

For the final claim, recall from our proof of (b)$\Rightarrow$(c) that $G$ is unitarily equivalent to a block diagonal matrix with blocks $\{H_\alpha\}_{\alpha\in\hat\Gamma}$.
Since $G^2=tG$ by assumption, it follows that each $H_\alpha$ has eigenvalues in $\{0,t\}$.
Next, since $\sum_{i\in[t]} G_{i,i}=tI$, then for each $\alpha\in\hat\Gamma$, we have
\[
\operatorname{tr}(H_\alpha)
=\sum_{i\in[t]} (H_\alpha)_{i,i}
=\sum_{i\in[t]} (FG_{i,i}F^{-1})_{\alpha,\alpha}
=(F tI F^{-1})_{\alpha,\alpha}
=t.
\]
It follows that $t$ is an eigenvalue of each $H_\alpha$ with multiplicity exactly $1$.
Finally, our proof of (c)$\Rightarrow$(a) constructs a $t$-generator $\Gamma$-harmonic collection $\{\rho(g)\psi_i\}_{g\in\Gamma,i\in[t]}$ of vectors with Gram matrix $G$, and since $\operatorname{rank}(H_\alpha)=1$ for every $\alpha\in\hat\Gamma$, it follows that these vectors reside in $\mathbb{C}^m$ and $\rho=\bigoplus_{\alpha\in\hat\Gamma}\alpha$ is equivalent to the regular representation of $\Gamma$, where the equivalence is given by $F$.
\end{proof}

Frequently, we are given a sequence $\Phi$ of vectors that isn't multi-generator harmonic, but rather switching equivalent to a multi-generator harmonic sequence of vectors.
If the underlying group is cyclic, then by Theorem~\ref{thm.t-gen C_m-harm detector} below, we can detect such switching equivalence from a particular \textit{automorphism} of $\Phi$, which we define in the following.

\begin{definition}
\label{def.automorphism}
An \textbf{automorphism} of $\Phi=\{\varphi_i\}_{i\in[n]}$ in $\mathbb{C}^d$ is an index permutation $\sigma\in S_n$ for which there exist unimodular scalars $c_1,\ldots,c_n\in\mathbb{C}$ such that for every $i,j\in[n]$, it holds that
\begin{equation}
\label{eq.automorphism def}
\langle \varphi_i,\varphi_j\rangle=\overline{c_i}c_j\langle \varphi_{\sigma(i)},\varphi_{\sigma(j)}\rangle.
\end{equation}
We also refer to such index permutations as automorphisms of the corresponding Gram matrix and, in the case of ETFs, the corresponding signature matrix.
\end{definition}

\begin{theorem}
\label{thm.t-gen C_m-harm detector}
Given a spanning set $\Phi=\{\varphi_i\}_{i\in[n]}$ for $\mathbb{C}^d$, the following are equivalent:
\begin{itemize}
\item[(a)]
$\Phi$ is switching equivalent to a $t$-generator $C_m$-harmonic collection of vectors.
\item[(b)]
There exists an automorphism of $\Phi$ of cycle type $m^t$.
\end{itemize}
\end{theorem}

\begin{proof}
For (a)$\Rightarrow$(b), suppose $\Phi$ is switching equivalent to $\{\rho(g)\psi_j\}_{g\in C_m,j\in[t]}$.
Then there exists a bijection $f\colon C_m\times[t]\to[n]$ and unimodular scalars $a_1,\ldots,a_n\in\mathbb{C}$ such that $\rho(g)\psi_j=a_{f(g,j)}\varphi_{f(g,j)}$ for every $g\in C_m$ and $j\in[t]$.
Letting $h$ denote a generator of $C_m$, then the permutation $(g,j)\mapsto(hg,j)$ on $C_m\times [t]$ induces a permutation $\sigma\in S_n$ of cycle type $m^t$.
Considering
\begin{align*}
\overline{a_{\sigma(f(g,j))}}a_{\sigma(f(g',j'))}
\langle\varphi_{\sigma(f(g,j))},\varphi_{\sigma(f(g',j'))}\rangle
&=\overline{a_{f(hg,j)}}a_{f(hg',j')}
\langle\varphi_{f(hg,j)},\varphi_{f(hg',j')}\rangle\\
&=\overline{a_{f(hg,j)}}a_{f(hg',j')}
\langle \overline{a_{f(hg,j)}}\rho(hg)\psi_j, \overline{a_{f(hg',j')}}\rho(hg')\psi_{j'}\rangle\\
&=\langle \rho(g)\psi_j, \rho(g')\psi_{j'}\rangle\\
&=\overline{a_{f(g,j)}}a_{f(g',j')}\langle\varphi_{f(g,j)}, \varphi_{f(g',j')}\rangle,
\end{align*}
it follows that $\sigma$ is an automorphism of $\Phi$ with unimodular scalars $c_{i}:=\overline{a_{i}}a_{\sigma(i)}$.

For (b)$\Rightarrow$(a), suppose $\sigma$ is an automorphism of $\Phi$ of cycle type $m^t$ with unimodular scalars $c_1,\ldots,c_n\in\mathbb{C}$.
Consider the graph with vertices $[n]$ and adjacency $i\leftrightarrow j$ whenever $j=\sigma(i)$ or $\langle\varphi_i,\varphi_j\rangle\neq0$.
The connected components of this graph are the equivalence classes of an equivalence relation $\sim$ on $[n]$.
We first prove an intermediate claim that the map 
\[
f\colon i\mapsto \prod_{k=0}^{m-1}c_{\sigma^k(i)}
\]
is constant over each equivalence class of $\sim$.
First, $f(\sigma(i))=f(i)$ by cyclically reindexing the factors of the product.
Next, we re-express each factor $c_{\sigma^k(j)}$ of $f(j)$.
To do so, we rewrite equation \eqref{eq.automorphism def} with $\sigma^k(i)$ and $\sigma^k(j)$ taking the place of $i$ and $j$, respectively:
\[
\langle\varphi_{\sigma^k(i)},\varphi_{\sigma^k(j)}\rangle
=\overline{c_{\sigma^k(i)}}c_{\sigma^k(j)}\langle\varphi_{\sigma^{k+1}(i)},\varphi_{\sigma^{k+1}(j)}\rangle.
\]
If $\langle\varphi_i,\varphi_j\rangle\neq0$, then we may isolate $c_{\sigma^k(j)}$ in the above equation to obtain
\begin{align*}
f(j)
=\prod_{k=0}^{m-1}c_{\sigma^k(j)}
&=\prod_{k=0}^{m-1}\frac{c_{\sigma^k(i)}\langle\varphi_{\sigma^k(i)},\varphi_{\sigma^k(j)}\rangle}{\langle\varphi_{\sigma^{k+1}(i)},\varphi_{\sigma^{k+1}(j)}\rangle}\\
&=\bigg(\prod_{k=0}^{m-1}c_{\sigma^k(i)}\bigg)\bigg(\frac{\prod_{k=0}^{m-1}\langle\varphi_{\sigma^k(i)},\varphi_{\sigma^k(j)}\rangle}{\prod_{k=0}^{m-1}\langle\varphi_{\sigma^{k+1}(i)},\varphi_{\sigma^{k+1}(j)}\rangle}\bigg)
=\prod_{k=0}^{m-1}c_{\sigma^k(i)}
=f(i).
\end{align*}
Then our intermediate claim follows by transitivity.

Let $A_1,\ldots,A_p\subseteq[n]$ denote the equivalence classes of $\sim$, and observe that the subspaces $V_\ell:=\operatorname{span}\{\varphi_i\}_{i\in A_\ell}$ are mutually orthgonal.
In fact, since $\Phi$ spans $\mathbb{C}^d$, we have $\mathbb{C}^d=V_1\oplus\cdots\oplus V_p$.
Since $\{\varphi_i\}_{i\in[n]}$ has the same Gram matrix as $\{c_i\varphi_{\sigma(i)}\}_{i\in[n]}$ by assumption, it follows that there exists $W\in U(d)$ such that $W\varphi_i=c_i\varphi_{\sigma(i)}$.
By the definition of $\sim$, then $W$ holds each $V_\ell$ invariant, i.e., $W=W_1\oplus\cdots\oplus W_p$ with $W_\ell\in U(V_\ell)$ for each $\ell\in[p]$.
Furthermore, for each $i\in A_\ell$, we have
\[
W_\ell\varphi_i
=c_i\varphi_{\sigma(i)},
\quad
W_\ell^2\varphi_i
=c_i W_\ell\varphi_{\sigma(i)}
=c_ic_{\sigma(i)}\varphi_{\sigma^2(i)},
\quad
\ldots
\quad
W_\ell^m\varphi_i
=f(i) \varphi_i.
\]
By our intermediate claim above, $f(i)$ is constant for $i\in A_\ell$.
Fix any unimodular scalar $\beta_\ell\in\mathbb{C}$ such that $\beta_\ell^m$ is this value of $f$ on $A_\ell$, define $U_\ell:=\overline{\beta_\ell}W_\ell$ for each $\ell\in[p]$, and put $U:=U_1\oplus\cdots\oplus U_p$.
Then for each $i\in A_\ell$, we have
\[
U_\ell^m\varphi_i
=\overline{\beta_\ell^m}W_\ell^m\varphi_i
=\overline{\beta_\ell^m}f(i) \varphi_i
=\varphi_i,
\]
and so $U^m=I$.
Letting $h$ denote a generator of $C_m$, then we may define $\rho\colon C_m\to U(d)$ by $\rho(h^k):=U^k$.
Then for any choice of representatives $i_1,\ldots,i_t\in[n]$ of the $t$ cycles in $\sigma$, it holds that $\Phi$ is switching equivalent to $\{\rho(g)\varphi_{i_j}\}_{g\in C_m,j\in[t]}$, which is $t$-generator $C_m$-harmonic.
\end{proof}

\begin{remark}
\label{rk.use case}
In our use case of Theorem~\ref{thm.t-gen C_m-harm detector}, $\Phi$ is an equiangular tight frame, and so $\langle \varphi_i,\varphi_j\rangle\neq0$ for all $i,j\in[n]$.
In this case, the proof of (b)$\Rightarrow$(a) delivers a more explicit result.
First, there necessarily exists a unimodular scalar $\beta$ such that $\prod_{k=0}^{m-1}c_{\sigma^k(i)}=\beta^m$ for all $i\in[n]$.
After selecting representatives $i_1,\ldots,i_t\in[n]$ of the $t$ cycles in $\sigma$, put
\[
a_{\sigma^\ell(i_j)}
:=\frac{1}{\beta^\ell}\prod_{k=0}^{\ell-1}c_{\sigma^k(i_j)}.
\]
Then $\{a_{\sigma^\ell(i_j)}\varphi_{\sigma^\ell(i_j)}\}_{\ell\in C_m,j\in[t]}$ is $t$-generator $C_m$-harmonic.
\end{remark}

In Example~\ref{ex.complex maximal}, we observed that equality in Gerzon's bound is achieved by complex $d$-circulant ETFs as a consequence of Zauner's conjecture.
There is also a real version of Gerzon's bound, namely, that a \textit{real} $d\times n$ ETF exists only if $n\leq \binom{d+1}{2}$.
Unlike the complex case, the real version of Gerzon bound equality is only known to be achieved for $d\in\{2,3,7,23\}$.
One might ask whether Gerzon bound equality in the real case is similarly achieved by real $\frac{d+1}{2}$-circulant ETFs.
While $d=2$ has the difficulty that $\frac{d+1}{2}=\frac{3}{2}$ is not an integer, one might half-jokingly suggest that
\[
\left[\begin{array}{rr|c}
\cos(\frac{\pi}{12})&-\sin(\frac{\pi}{12})\quad&\quad\frac{1}{\sqrt{2}}\quad\\
-\sin(\frac{\pi}{12})&\cos(\frac{\pi}{12})\quad&\frac{1}{\sqrt{2}}
\end{array}\right]
\]
is a real $\frac{3}{2}$-circulant $2\times 3$ ETF since it consists of a circulant matrix and half of another. 
Next, real $\frac{d+1}{2}$-circulant ETFs are given for $d=3$ and $d=7$ in Examples~\ref{ex.3x6} and~\ref{ex.steiner}, respectively.
In what follows, we use the technology developed in this section to resolve the remaining case of $d=23$.

\begin{example}
There exists a real $23\times 276$ ETF whose automorphism group is the third Conway group $\operatorname{Co}_3 \leq S_{276}$ in its doubly transitive action on $276$ points~\cite{Taylor:92,IversonM:24}.
Using GAP~\cite{GAP:online}, we observe that $\operatorname{Co}_3$ contains a permutation of cycle type $23^{12}$.
By Theorem~\ref{thm.t-gen C_m-harm detector}, it follows that this ETF is switching equivalent to a $12$-generator $C_{23}$-harmonic collection of vectors.
Furthermore, the Gram matrix of this ETF satisfies the conditions in \eqref{eq.reg repn detector}, and so Theorem~\ref{thm.t-gen Gamma-harm detector} gives that there exists a $12$-circulant $23\times 276$ ETF.
One may follow the proofs of these theorems to construct the ETF in its $12$-circulant form.
The resulting ETF turns out to be real and appears as an ancillary file with the arXiv version of this paper.
\end{example}

\section{Infinite families of complex \(2\)-circulant ETFs}
\label{sec.infinite families}

In this section, we identify multiple infinite families of $2$-circulant ETFs that we present as evidence in favor of the medium $d\times 2d$ conjecture (Conjecture~\ref{conj.medium d by 2d}).

\begin{theorem}
\label{thm.2-circ ETF families}
Every $d\times 2d$ ETF in the following table is switching equivalent to a $2$-circulant ETF:
\begin{center}
\begin{tabular}{lllll}
\hline
construction & $2d$ & $q$ & small values of $d$ & location\\
\hline\hline
$G_q+1$ & $q+1$ & $1\bmod 4$ prime power & $3$, $5$, $7$, $9$, $13$ & Example~\ref{ex.paley conference matrices} \\
&& $3\bmod 4$ prime power & $2$, $4$, $6$, $10$, $12$ & Example~\ref{ex.paley conference matrices}  \\
\hline
$2\cdot G_q$ & $2q$ & $1\bmod 4$ prime & $5$, $13$, $17$, $29$, $37$ & Example~\ref{eq.double paley graph} \\
&& $3\bmod 4$ prime & $3$, $7$, $11$, $19$, $23$ & Example~\ref{ex.doubling the renes-strohmer ETF} \\
\hline
$2\cdot(G_q+1)$ & $2(q+1)$ & $1\bmod 4$ prime power & $6$, $10$, $14$, $18$, $26$ & Example~\ref{ex.doubling the paley ETF} \\
&& $3\bmod 4$ prime power & $4$, $8$, $12$, $20$, $24$ & Example~\ref{ex.doubling the paley ETF} \\
\hline
\end{tabular}
\end{center}
\end{theorem}

The first dimension not covered by the constructions in Theorem~\ref{thm.2-circ ETF families} is $d=33$.
When $q=p^k$ with $p$ prime, it more generally holds that $2\cdot G_q$ ETFs are $2$-generator $C_p^k$-harmonic with the regular representation. 
Considering $G_q+1$ and $2\cdot(G_q+1)$ both appear in the above table, one might wonder whether the double of a $2$-circulant ETF is necessarily (switching equivalent to) a $2$-circulant ETF.
Sadly, this is not the case: empirically, doubling a randomly generated $2$-circulant $5\times 10$ ETF produces a $10\times 20$ ETF whose automorphism group is isomorphic to $C_5$.

\begin{proof}[Proof of Theorem~\ref{thm.2-circ ETF families}]
We start by considering $2\cdot G_q$.
Since $q$ is prime by assumption, the Paley graph is circulant, and so the signature matrix of $2\cdot G_q$ has circulant blocks.
Furthermore, in both~\eqref{eq.doubled ETF signature matrix} and~\eqref{eq.doubling conference graph to signature matrix}, the diagonal blocks of the signature matrix sum to the zero matrix.
By Theorem~\ref{thm.t-gen Gamma-harm detector}, it follows that the ETF is $2$-circulant in the appropriate basis.
(In fact, Theorem~\ref{thm.short fat repn of doubling} gives an explicit $2$-circulant representation in the $q\equiv 1\bmod 4$ case.)

For $G_q+1$, note that by Proposition~\ref{prop.paley conference matrices are the equivalent}, we are free to work with any conference matrix constructed in Example~\ref{ex.symplectic form construction}.
Given an odd prime power $q$, we consider the $2$-dimensional vector space $\mathbb{F}_{q^2}$ over $\mathbb{F}_q$ with symplectic form $[\cdot,\cdot]\colon\mathbb{F}_{q^2}\times\mathbb{F}_{q^2}\to\mathbb{F}_q$ defined by
\[
[x,y]
=\zeta^{\frac{q+1}{2}}(xy^q-yx^q),
\]
where $\zeta$ denotes a generator of $\mathbb{F}_{q^2}^\times$.
Indeed, $[x,y]\in\mathbb{F}_q$ follows from verifying $[x,y]^q=[x,y]$, the form is $\mathbb{F}_q$-bilinear since $z\mapsto z^q$ is $\mathbb{F}_q$-linear, and the form is clearly alternating.
To establish nondegeneracy, given any nonzero $x\in\mathbb{F}_{q^2}$, one may select $y\in\mathbb{F}_{q^2}^\times$ from a different coset of $\mathbb{F}_q^\times$ to get
\[
[x,y]
=\zeta^{\frac{q+1}{2}}xy(y^{q-1}-x^{q-1})
\neq0,
\]
where the last step follows from the facts that $x$ and $y$ are both nonzero and $z\mapsto z^{q-1}$ is a homomorphism on $\mathbb{F}_{q^2}^\times$ with kernel $\mathbb{F}_q^\times$.
In what follows, we repeatedly make use of the fact that
\begin{equation}
\label{eq.scalars pull out}
[zx,zy]
=z^{q+1}[x,y]
\end{equation}
for any $x,y,z\in\mathbb{F}_{q^2}$.

We will carefully select nonzero $t_1,\ldots,t_{q+1}\in\mathbb{F}_{q^2}$ so that the signature matrix $S$ defined by
\[
S_{i,j}
=\omega\chi([t_j,t_j]),
\qquad
\omega
:=\left\{\begin{array}{cl}
1&\text{if }q\equiv 1\bmod 4\\
\mathrm{i}&\text{if }q\equiv 3\bmod 4
\end{array}\right.
\]
has an easily identified automorphism (\`{a} la Definition~\ref{def.automorphism}) of cycle type $(\frac{q+1}{2})^2$.
Then by Theorem~\ref{thm.t-gen C_m-harm detector}, any full-rank $\Phi$ with Gram matrix $I+\frac{1}{\sqrt{q}}S$ is switching equivalent to a $t$-generator $C_d$-harmonic collection of vectors.
To this end, let $L\colon \mathbb{F}_{q^2}\to\mathbb{F}_{q^2}$ denote the $\mathbb{F}_{q}$-linear map defined by $L(x)=\zeta^{-q+1}x$, and note that $L$ preserves our symplectic form by equation~\eqref{eq.scalars pull out}.
We first consider how $L$ acts on the $1$-dimensional subspaces of $\mathbb{F}_{q^2}$.
Since $\zeta^{q+1}$ generates $\mathbb{F}_q^\times$, we may index these subspaces as
\[
\ell_k
:=\operatorname{span}_{\mathbb{F}_q}\{\zeta^k\}
=\{0\}\cup\{\zeta^{k+j(q+1)}:j\in\mathbb{Z}\}
\]
so that $\ell_k=\ell_{k'}$ precisely when $k\equiv k'\bmod q+1$, then $L$ maps $\ell_k$ to $\ell_{k+2}$.
It follows that $L$ induces a permutation with cycle type $(\frac{q+1}{2})^2$ of the lines.
This suggests taking line representatives indexed by $\{0,\ldots,\frac{q-1}{2}\}\times\{0,1\}$, namely
\[
t_{k,0}:=L^k1,
\qquad
t_{k,1}:=L^k\zeta,
\qquad
0\leq k<\tfrac{q+1}{2}.
\]
Then $Lt_{k,\varepsilon}=t_{k+1,\varepsilon}$ for $0\leq k<\tfrac{q-1}{2}$ and $\varepsilon\in\{0,1\}$, but since $L^{\frac{q+1}{2}}=-\operatorname{id}$, we also have $Lt_{\frac{q-1}{2},\varepsilon}=-t_{0,\varepsilon}$.
We notate these relationships by writing
\[
Lt_{k,\varepsilon}
=\alpha_kt_{k+1,\varepsilon}
\]
with $\alpha_k\in\{\pm1\}\subseteq\mathbb{F}_q^\times$, and with the understanding that $k+1$ is to be interpreted modulo $\frac{q+1}{2}$.
Then
\[
[t_{i,\varepsilon},t_{j,\delta}]
=[Lt_{i,\varepsilon},Lt_{j,\delta}]
=[\alpha_i t_{i+1,\varepsilon},\alpha_j t_{j+1,\delta}]
=\alpha_i\alpha_j[ t_{i+1,\varepsilon}, t_{j+1,\delta}].
\]
Let $\sigma$ denote the permutation of $\{0,\ldots,\frac{q-1}{2}\}\times\{0,1\}$ defined by $\sigma(k,\varepsilon)=(k+1,\varepsilon)$, and put $c_{k,\varepsilon}:=\chi(\alpha_k)\in\{\pm1\}$.
Then
\[
S_{(i,\varepsilon),(j,\delta)}
=\omega\chi([t_{i,\varepsilon},t_{j,\delta}])
=\chi(\alpha_i)\chi(\alpha_j)\omega\chi([t_{\sigma(i,\varepsilon)},t_{\sigma(j,\delta)}])
=c_{i,\varepsilon}c_{j,\delta} S_{\sigma(i,\varepsilon),\sigma(j,\delta)}
=\overline{c_{i,\varepsilon}}c_{j,\delta} S_{\sigma(i,\varepsilon),\sigma(j,\delta)}.
\]
It follows that $\sigma$ is an automorphism of $S$, as desired.

At this point, we have that any full-rank $\Phi=\{\varphi_{k,\varepsilon}\}$ with Gram matrix $I+\frac{1}{\sqrt{q}}S$ is switching equivalent to a $2$-generator $C_d$-harmonic collection of vectors, where $d=\frac{q+1}{2}$.
Furthermore, since $c_{k,0}=c_{k,1}$ for each $k$, Remark~\ref{rk.use case} produces unimodular scalars $\{a_{k,\varepsilon}\}$ with $a_{k,0}=a_{k,1}$ for each $k$ such that $\{a_{k,\varepsilon}\varphi_{k,\varepsilon}\}$ is $2$-generator $C_d$-harmonic.
We conclude by using Theorem~\ref{thm.t-gen Gamma-harm detector} to show that the underlying representation is equivalent to the regular representation of $C_d$.
To this end, letting $D$ denote the diagonal matrix with diagonal entries $\{a_{k,\varepsilon}\}$, it suffices to verify that the signature matrix $D^*SD\in(\mathbb{C}^{C_d\times C_d})^{\{0,1\}\times\{0,1\}}$ satisfies $(D^*SD)_{1,1}=-(D^*SD)_{0,0}$.
Since $\zeta^{q+1}$ generates $\mathbb{F}_q^\times$, we have $\chi(\zeta^{q+1})=-1$, and by writing $t_{i,1}=\zeta t_{i,0}$, we have
\begin{align*}
((D^*SD)_{1,1})_{i,j}
&=\overline{a_{i,1}}a_{j,1}\omega\chi([t_{i,1},t_{j,1}])
=\overline{a_{i,0}}a_{j,0}\omega\chi([\zeta t_{i,0},\zeta t_{j,0}])\\
&=\overline{a_{i,0}}a_{j,0}\omega\chi(\zeta^{q+1}[t_{i,0},t_{j,0}])
=-\overline{a_{i,0}}a_{j,0}\omega\chi([t_{i,0},t_{j,0}])
=-((D^*SD)_{0,0})_{i,j}.
\end{align*}
Overall, the Gram matrix $I+\frac{1}{\sqrt{q}}S$ is switching equivalent to the Gram matrix of a $2$-circulant ETF.

Finally, for $2\cdot(G_q+1)$, we again view $\mathbb{F}_{q^2}$ as the same symplectic $\mathbb{F}_q$-vector space.
Selecting line representatives $t_j:=\zeta^j$ for $j\in\{0,\ldots,q\}$, then the corresponding signature matrix is given by
\begin{equation}
\label{eq.sig of 2 Gq+1 pt 1}
S_{(i,\varepsilon),(j,\delta)}
=(-1)^{\varepsilon\delta}\omega\chi([t_i,t_j])+\mathbf{1}_{\{i=j\}}\mathbf{1}_{\{\varepsilon\neq\delta\}}\mathrm{i}^{\delta-\varepsilon},
\qquad
(i,\varepsilon),(j,\delta)\in\{0,\ldots,q\}\times\{0,1\},
\end{equation}
where $\omega$ is $1$ if $q\equiv 1\bmod 4$ and $\mathrm{i}$ if $q\equiv 3\bmod 4$.
Consider the $\operatorname{F}_q$-linear map $L\colon \mathbb{F}_{q^2}\to\mathbb{F}_{q^2}$ defined by $L(x) = \zeta x$.
While $L$ is not symplectic, it does satisfy $[Lx,Ly]=\zeta^{q+1}[x,y]$ by equation~\eqref{eq.scalars pull out}.
Also, $L$ maps $\ell_k$ to $\ell_{k+1}$, and in particular, $Lt_k=\alpha_k t_{k+1}$, where $\alpha_k=1$ for $k\in\{0,\ldots,q-1\}$ and $\alpha_q=\zeta^{q+1}$.
We claim that the permutation $\sigma\colon(i,\varepsilon)\mapsto(i+1,1-\varepsilon)$ is an automorphism of $S$ \`{a} la Definition~\ref{def.automorphism}.
First, since $\zeta^{q+1}$ generates $\mathbb{F}_q^\times$, we have $\chi(\zeta^{q+1})=-1$, and so
\[
-\chi([t_i,t_j])
=\chi([Lt_i,Lt_j])
=\chi([\alpha_i Lt_{i+1},\alpha_j Lt_{j+1}])
=\chi(\alpha_i)\chi(\alpha_j)\chi([t_{i+1},t_{j+1}]).
\]
Substituting this identity into \eqref{eq.sig of 2 Gq+1 pt 1} gives
\[
S_{(i,\varepsilon),(j,\delta)}
=(-1)^{\varepsilon\delta}\omega\Big(-\chi(\alpha_i)\chi(\alpha_j)\chi([t_{i+1},t_{j+1}])\Big)+\mathbf{1}_{\{i=j\}}\mathbf{1}_{\{\varepsilon\neq\delta\}}\mathrm{i}^{\delta-\varepsilon}.
\]
Next, we apply the fact that $(-1)^{(1-\varepsilon)(1-\delta)}
=(-1)^{1-\varepsilon-\delta+\varepsilon\delta}
=-(-1)^\varepsilon(-1)^\delta(-1)^{\varepsilon\delta}$:
\begin{equation}
\label{eq.sig of 2 Gq+1 pt 2}
S_{(i,\varepsilon),(j,\delta)}
=(-1)^\varepsilon\chi(\alpha_i)\cdot(-1)^\delta\chi(\alpha_j)\cdot(-1)^{(1-\varepsilon)(1-\delta)}\omega\chi([t_{i+1},t_{j+1}])+\mathbf{1}_{\{i=j\}}\mathbf{1}_{\{\varepsilon\neq\delta\}}\mathrm{i}^{\delta-\varepsilon}.
\end{equation}
To simplify the second term above, note that if $\varepsilon\neq\delta$ so that $\delta-\varepsilon\in\{\pm1\}$, then 
\[
\mathrm{i}^{(1-\delta)-(1-\varepsilon)}
=\mathrm{i}^{-\delta+\varepsilon}
=\overline{\mathrm{i}^{\delta-\varepsilon}}
=-\mathrm{i}^{\delta-\varepsilon}.
\]
Also, if $i=j$ and $\varepsilon\neq\delta$, then 
\[
(-1)^\varepsilon\chi(\alpha_i)\cdot(-1)^\delta\chi(\alpha_j)
=-1.
\]
It follows that the second term in \eqref{eq.sig of 2 Gq+1 pt 2} is
\[
(-1)^\varepsilon\chi(\alpha_i)\cdot(-1)^\delta\chi(\alpha_j)\cdot
\mathbf{1}_{\{i+1=j+1\}}\mathbf{1}_{\{1-\varepsilon\neq1-\delta\}}\mathrm{i}^{(1-\delta)-(1-\varepsilon)}.
\]
Overall, taking $c_{i,\varepsilon}:=(-1)^\varepsilon\chi(\alpha_i)$, it holds that
\[
S_{(i,\varepsilon),(j,\delta)}
=\overline{c_{i,\varepsilon}}c_{j,\delta}S_{\sigma(i,\varepsilon),\sigma(j,\delta)}.
\]
Thus, $\sigma$ is an automorphism of $S$.
Considering $\sigma$ has exactly two cycles, namely,
\begin{align*}
&(0,0)\mapsto(1,1)\mapsto(2,0)\mapsto\cdots\mapsto(q,1)\mapsto(0,0),\\
&
(0,1)\mapsto(1,0)\mapsto(2,1)\mapsto\cdots\mapsto(q,0)\mapsto(0,1),
\end{align*}
it follows that the $2\cdot(G_q+1)$ ETF is switching equivalent to a $2$-generator $C_d$-harmonic collection of vectors, where $d=q+1$.
It remains to apply Theorem~\ref{thm.t-gen Gamma-harm detector} to show that the underlying representation is equivalent to the regular representation of $C_d$.
In particular, letting $\{a_{i,\varepsilon}\}$ denote the scalars used in the switching equivalence, and letting $D$ denote the diagonal matrix with diagonal entries $\{a_{i,\varepsilon}\}$, it suffices to verify that
\begin{equation}
\label{eq.reg repn for 2.(Gq+1)}
(D^*SD)_{(0,1),\sigma^k(0,1)}
=-(D^*SD)_{(0,0),\sigma^k(0,0)}
\end{equation}
for every $k$.
Notice that $c_{i,1-\varepsilon}=-c_{i,\varepsilon}$, and so by Remark~\ref{rk.use case}, we have $a_{i,1-\varepsilon}=(-1)^i a_{i,\varepsilon}$.
We will apply this fact to verify \eqref{eq.reg repn for 2.(Gq+1)} in cases.
First, if $k$ is even, then
\begin{align*}
(D^*SD)_{(0,1),\sigma^k(0,1)}
&=(D^*SD)_{(0,1),(k,1)}\\
&=\overline{a_{0,1}}a_{k,1}\big(-\omega\chi([t_0,t_k])\big)\\
&=-\overline{a_{0,0}}a_{k,0}\omega\chi([t_0,t_k])
=-(D^*SD)_{(0,0),(k,0)}
=-(D^*SD)_{(0,0),\sigma^k(0,0)}.
\end{align*}
Similarly, if $k$ is odd, then
\begin{align*}
(D^*SD)_{(0,1),\sigma^k(0,1)}
&=(D^*SD)_{(0,1),(k,0)}\\
&=\overline{a_{0,1}}a_{k,0}\omega\chi([t_0,t_k])\\
&=-\overline{a_{0,0}}a_{k,1}\omega\chi([t_0,t_k])
=-(D^*SD)_{(0,0),(k,1)}
=-(D^*SD)_{(0,0),\sigma^k(0,0)}.
\end{align*}
Overall, the $2\cdot(G_q+1)$ ETF is switching equivalent to a $2$-circulant ETF.
\end{proof}

\section{Continua of complex \(2\)-circulant ETFs}
\label{sec.continua}

In this section, we pose an even stronger version of the $d\times 2d$ conjecture that predicts the existence of a $\Theta(d)$-dimensional manifold of complex $2$-circulant $d\times 2d$ ETFs.
This seems plausible because, as we show, the defining equations reduce to fewer polynomial constraints than variables.
In what follows, we zero-index vectors in $\mathbb{C}^d$, $T$ denotes the circular translation operator defined by $e_i\mapsto e_{i+1\bmod d}$ and extending linearly, and $\hat{x}$ denotes the discrete Fourier transform of $x$ defined by $\hat{x}_k=\sum_{j=0}^{d-1} x_j e^{-2\pi \mathrm{i}jk/d}$.

\begin{lemma}
\label{lem.2-circ constraints}
Suppose $x,y\in\mathbb{C}^d$ have the property that there exist $\alpha,\beta\in\mathbb{R}$ such that
\begin{align}
\label{eq.2circ unit norm} \|x\|^2=\|y\|^2 &= 1,\\
\label{eq.2circ tight} |\hat{x}_k|^2+|\hat{y}_k|^2 &= \alpha \qquad \forall k\in\{0,\ldots,d-1\},\\
\label{eq.2circ equi 1} |\langle x,T^kx\rangle|^2 &= \beta \qquad \forall k\in\{1,\ldots,\lfloor\tfrac{d}{2}\rfloor\},\\
\label{eq.2circ equi 2} |\langle x,T^ky\rangle|^2 &= \beta \qquad \forall k\in\{0,\ldots,d-1\}.
\end{align}
Then $x$ and $y$ together generate a $2$-circulant $d\times 2d$ ETF.
\end{lemma}

\begin{proof}
Let $C_x$ and $C_y$ denote the circulant matrices with first columns $x$ and $y$, respectively.
We wish to show that $\Phi=\left[\begin{array}{cc} C_x & C_y \end{array}\right]$ is an ETF.
First, the columns of $\Phi$ have unit norm by \eqref{eq.2circ unit norm}.
Next, \eqref{eq.2circ tight} implies that every eigenvalue of $\Phi\Phi^*=C_x^{}C_x^*+C_y^{}C_y^*$ equals $\alpha$, meaning $\Phi\Phi^*=\alpha I$, i.e., $\Phi$ is tight.
It remains to verify that $\Phi$ is equiangular.
First, for any $i,j\in\mathbb{Z}$ with $i\not\equiv j \bmod d$, it holds that
\[
|\langle T^ix, T^jx \rangle|^2
=|\langle x, T^{j-i}x \rangle|^2
=|\langle x, T^{i-j}x \rangle|^2,
\]
and either $j-i$ or $i-j$ resides in $\{1,\ldots,\lfloor\tfrac{d}{2}\rfloor\}$ modulo $d$.
As such, the translations of $x$ are $\beta$-equiangular by \eqref{eq.2circ equi 1}.
Similarly, \eqref{eq.2circ equi 2} implies that the translations of $x$ are $\beta$-equiangular with the translations of $y$.
Finally, \eqref{eq.2circ tight} gives that
\[
\alpha I
=\Phi\Phi^*
=C_x^{}C_x^*+C_y^{}C_y^*
=C_x^*C_x^{}+C_y^*C_y^{},
\]
where the last step applies the fact that circulant matrices commute.
Thus, the off-diagonal entries of $C_y^*C_y^{}$ are opposite the off-diagonal entries of $C_x^*C_x^{}$.
Since the translations of $x$ are $\beta$-equiangular by \eqref{eq.2circ equi 1}, it follows that the translations of $y$ are also $\beta$-equiangular.
\end{proof}

The following tables count the real variables and constraints in Lemma~\ref{lem.2-circ constraints}:
\begin{center}
\begin{tabular}{cc}
object & variables \\ \hline
$x$ & $2d$ \\
$y$ & $2d$ \\
$\alpha$ & 1 \\
$\beta$ & 1 
\end{tabular}
\qquad\qquad
\begin{tabular}{cc}
equation & constraints \\ \hline
\eqref{eq.2circ unit norm} & $2$ \\
\eqref{eq.2circ tight} & $d$ \\
\eqref{eq.2circ equi 1} & $\lfloor\frac{d}{2}\rfloor$ \\
\eqref{eq.2circ equi 2} & $d$
\end{tabular}
\end{center}
We expect the dimension of the set of $(x,y)$'s that generate $2$-circulant $d\times 2d$ ETFs to equal the dimension of the set of $(x,y,\alpha,\beta)$'s that satisfy \eqref{eq.2circ unit norm}--\eqref{eq.2circ equi 2}, since $\alpha$ and $\beta$ are determined by $x$ and $y$.
Furthermore, we naively expect the dimension of this set of $(x,y,\alpha,\beta)$'s to equal the total number of variables minus the total number of constraints, namely $\lceil\frac{3d}{2}\rceil$.
In addition, one might mod out by switching equivalence, which is easier to account for in the Fourier basis that diagonalizes the translation operator.
In this basis, one may phase any of the $d$ coordinates to obtain another $2$-circulant ETF.
One may also phase either $x$ or $y$, though phasing both $x$ and $y$ with a common phase is equivalent to phasing all $d$ coordinates with that phase.
After modding out by this $(d+1)$-dimensional torus action, we (again, naively) predict that the set of inequivalent $2$-circulant $d\times 2d$ ETFs has dimension
\begin{equation}
\label{eq.predicted dimension}
\lceil\tfrac{3d}{2}\rceil - (d + 1).
\end{equation}
Perhaps surprisingly, this suggests that the dimension of the solution set grows with $d$.
The following examples evaluate this prediction in small dimensions.

\begin{example}
When $d=2$, the naive calculation \eqref{eq.predicted dimension} predicts a $0$-dimensional set of inequivalent $2$-circulant ETFs.
In his thesis~\cite{Zauner:99}, Zauner presents a $2$-circulant $2\times 4$ ETF with signature matrix
\[
\left[\begin{array}{cc|cc}
0&\phantom{-}1&\phantom{-}1&-\mathrm{i}\\
1&\phantom{-}0&-\mathrm{i}&\phantom{-}1\\\hline
1&\phantom{-}\mathrm{i}&\phantom{-}0&-1\\
\mathrm{i}&\phantom{-}1&-1&\phantom{-}0
\end{array}\right].
\]
(Such a $2$-circulant ETF also arises by applying Theorem~\ref{thm.2-circ ETF families} to $G_3+1$.)
It turns out that there is, indeed, a $0$-dimensional set of inequivalent $2$-circulant $2\times 4$ ETFs, and in fact, the $2\times 4$ ETF is \textit{unique} up to switching equivalence (even without insisting on $2$-circulant structure).
To see this, note that a $2\times 4$ matrix with columns $\varphi_1,\ldots,\varphi_4$ is an ETF precisely when $\{\varphi_1\varphi_1^*,\ldots,\varphi_4\varphi_4^*\}$ is the vertex set of a regular tetrahedon in the so-called \textit{Bloch sphere}, namely, the $2$-sphere in the affine space of unit-trace self-adjoint $2\times 2$ matrices that is centered at $\frac{1}{2}I$ and has Frobenius-radius $\frac{1}{\sqrt{2}}$.
Then any two $2\times 4$ ETFs are switching equivalent since one may rotate any such regular tetrahedron into another by applying an appropriate $2\times 2$ unitary transformation to the ETF vectors.
(Here, we are making use of a well-known isomorphism between $PU(2)$ and $SO(3)$.)
\end{example}

\begin{example}
When $d=3$, the naive calculation \eqref{eq.predicted dimension} predicts a $1$-dimensional set of inequivalent $2$-circulant ETFs.
Et-Taoui presented a $1$-dimensional set of $3\times 6$ ETFs (see Subsection~\ref{subsec.other etfs of paley size}), and we empirically observe that these are switching equivalent to $2$-circulant ETFs.
As an explicit alternative, for any unimodular $\alpha\in\mathbb{C}$, the following is the signature matrix of a $2$-circulant $3\times 6$ ETF:
\[
\left[\begin{array}{ccc|ccc}
\phantom{-}0&\phantom{-}\alpha&\phantom{-}\overline\alpha & \phantom{-}1&\phantom{-}\alpha&-\overline\alpha \\
\phantom{-}\overline\alpha&\phantom{-}0&\phantom{-}\alpha & -\overline\alpha&\phantom{-}1&\phantom{-}\alpha \\
\phantom{-}\alpha&\phantom{-}\overline\alpha&\phantom{-}0 & \phantom{-}\alpha&-\overline\alpha&\phantom{-}1 \\ \hline
\phantom{-}1&-\alpha&\phantom{-}\overline\alpha & \phantom{-}0&-\alpha&-\overline\alpha \\
\phantom{-}\overline\alpha&\phantom{-}1&-\alpha & -\overline\alpha&\phantom{-}0&-\alpha \\
-\alpha&\phantom{-}\overline\alpha&\phantom{-}1 & -\alpha&-\overline\alpha&\phantom{-}0
\end{array}\right].
\]
Furthermore, generic members of this $1$-dimensional set are inequivalent since they exhibit distinct triple products~\cite{ChienW:16}.
This confirms our $1$-dimensional prediction above.
\end{example}

\begin{example}
When $d=4$, the naive calculation \eqref{eq.predicted dimension} predicts a $1$-dimensional set of inequivalent $2$-circulant ETFs.
Applying Theorem~\ref{thm.2-circ ETF families} to $G_7+1$ produces the following signature matrix:
\[
\left[
\begin{array}{llll|llll}
\phantom{-}0& -\omega^3 & -1 & \phantom{-}\omega & \phantom{-}\mathrm{i} & \phantom{-}\omega^3 & -1 & -\omega \\
\phantom{-}\omega & \phantom{-}0 & -\omega^3 & -1 & -\omega & \phantom{-}\mathrm{i} & \phantom{-}\omega^3 & -1 \\
-1 & \phantom{-}\omega & \phantom{-}0 & -\omega^3 & -1 & -\omega & \phantom{-}\mathrm{i} & \phantom{-}\omega^3 \\
-\omega^3 & -1 & \phantom{-}\omega & \phantom{-}0 & \phantom{-}\omega^3 & -1 & -\omega & \phantom{-}\mathrm{i} \\ \hline
-\mathrm{i} & \phantom{-}\omega^3 & -1 & -\omega & \phantom{-}0 & \phantom{-}\omega^3 & \phantom{-}1 & -\omega \\
-\omega & -\mathrm{i} & \phantom{-}\omega^3 & -1 & -\omega & \phantom{-}0 & \phantom{-}\omega^3 & \phantom{-}1 \\
-1 & -\omega & -\mathrm{i} & \phantom{-}\omega^3 & \phantom{-}1 & -\omega & \phantom{-}0 & \phantom{-}\omega^3 \\
\phantom{-}\omega^3 & -1 & -\omega & -\mathrm{i} & \phantom{-}\omega^3 & \phantom{-}1 & -\omega & \phantom{-}0
\end{array}
\right],
\]
where $\omega$ denotes the primitive eighth root of unity $e^{\pi\mathrm{i}/4}$.
Despite our $1$-dimensional prediction, we have not found a $2$-circulant $4\times 8$ ETF that is inequivalent to the one above.
In fact, we believe the $4\times 8$ ETF is unique up to switching equivalence (even without insisting on $2$-circulant structure).
As evidence for this claim, we conducted the following experiment:
From a random initialization, run $10,000$ iterations of the alternating projections routine in~\cite{TroppDHS:05} to obtain the Gram matrix of a numerical ETF, then pass to the signature matrix of phases, normalize the phases in the first row and column to be all $1$s, and store the largest absolute real part of the entries in the remaining $7\times 7$ core.
This quantifies the extent to which the numerical ETF is switching equivalent to a $4\times 8$ ETF obtained as in Subsection~\ref{subsec.classic constructions} from a skew-symmetric $8\times 8$ conference matrix (which is known to be unique up to switching~\cite{OEIS:online}).
After $1,000$ trials of this experiment, the largest deviation we encountered was less than $0.05$, which is quite small considering alternating projections converges slowly in this complex setting.
Furthermore, in all of our trials, rounding to the nearest signature matrix with imaginary core resulted in the signature matrix of an ETF.
\end{example}

\begin{conjecture}
\label{conj.4}
There exists a unique complex $4\times 8$ ETF up to switching equivalence.
\end{conjecture}

We attempted to prove Conjecture~\ref{conj.4} by reducing to a Gr\"{o}bner basis calculation as in~\cite{Szollosi:14}, but our naive implementation resulted in out-of-memory errors.
While the $d=4$ case appears to defy the predicted dimension of $2$-circulant ETFs, we believe that this case is sporadic, as stated in the following conjecture.

\begin{conjecture}[strong $d\times 2d$ conjecture]
\label{conj.strong d by 2d}
For each $d\neq4$, there exists a real smooth submanifold $M\subseteq(\mathbb{C}^d)^2$ of dimension $\lceil\frac{3d}{2}\rceil$ such that every member of $M$ generates a $2$-circulant $d\times 2d$ ETF.
\end{conjecture}

\begin{theorem}
\label{thm.strong d by 2d for small d}
The strong $d\times 2d$ conjecture holds for all $d\leq165$.
\end{theorem}

To prove Theorem~\ref{thm.strong d by 2d for small d}, we first applied local methods to find, for each $d\leq 165$, a $2$-circulant matrix of size $d\times 2d$ that is very nearly an ETF.
Finding such numerical $2$-circulant ETFs was remarkably easy in practice, and in fact, we observed that local methods converge from any random initialization.
We then used the following variant of the Newton--Kantorovich technique laid out by Cohn, Kumar, and Minton in~\cite{CohnKM:16} to establish the existence of a large manifold of exact $2$-circulant ETFs that is close to each of our numerical ETFs.
Here, the \textit{total degree} of a polynomial map $f\colon\mathbb{R}^m\to\mathbb{R}^n$ is the largest degree of its coordinate functions $f_1,\ldots,f_n\colon\mathbb{R}^m\to\mathbb{R}$, and $|f|:=\max_i|f_i|$, where $|f_i|$ is the sum of the absolute values of the coefficients of $f_i$.

\begin{theorem}
\label{thm.secant jacobian}
Fix a polynomial $f\colon\mathbb{R}^m\to\mathbb{R}^n$ of total degree $d$ and a point $x_0\in\mathbb{R}^m$.
Given $\delta>0$, consider the $\delta$-secant approximation $S\colon\mathbb{R}^m\to\mathbb{R}^n$ of the Jacobian of $f$ at $x_0$, defined by
\[
e_j
\mapsto
\frac{f(x_0+\delta e_j)-f(x_0)}{\delta}
\]
and extending linearly.
Denote $\tilde\alpha:=\alpha\max\{1,\|x_0\|_\infty+\alpha\}^{d-2}$ for any $\alpha>0$, and suppose there exists $\varepsilon>0$ and a linear map $T\colon \mathbb{R}^n\to\mathbb{R}^m$ such that 
\[
\|S\circ T-\operatorname{id}_{\mathbb{R}^n}
\|_{\infty\to\infty}
+(\tfrac{1}{2}\tilde\delta+\tilde\varepsilon)|f|d(d-1)\|T\|_{\infty\to\infty}
<1-\tfrac{1}{\varepsilon}\|T\|_{\infty\to\infty}\|f(x_0)\|_\infty.
\]
Then there exists $x^\star\in \mathbb{R}^m$ such that $f(x^\star)=0$ and $\|x^\star-x_0\|_\infty<\varepsilon$.
Furthermore, $f^{-1}(0)$ is locally a manifold of dimension $m-n$ near $x^\star$.
\end{theorem}

Theorem~\ref{thm.secant jacobian} is a version of Newton--Kantorovich that avoids computing an exact Jacobian. 
Performance may suffer, but this derivative-free alternative is much easier to implement in practice (at least when automatic differentiation is unavailable). 
Theorem~\ref{thm.strong d by 2d for small d} might be improved by computing exact Jacobians, but the primary bottleneck seems to be that the fast Fourier transform has no interval arithmetic implementation in Mathematica.

\begin{proof}[Proof of Theorem~\ref{thm.secant jacobian}]
By Corollary~3.4 in~\cite{CohnKM:16}, it suffices to show that
\[
\|Df(x_0)\circ T-\operatorname{id}_{\mathbb{R}^n}
\|_{\infty\to\infty}
+\tilde\varepsilon|f|d(d-1)\|T\|_{\infty\to\infty}
<1-\tfrac{1}{\varepsilon}\|T\|_{\infty\to\infty}\|f(x_0)\|_\infty,
\]
where $Df(x_0)\colon\mathbb{R}^m\to\mathbb{R}^n$ denotes the Jacobian of $f$ at $x_0$.
Considering
\begin{align*}
\|Df(x_0)\circ T-\operatorname{id}_{\mathbb{R}^n}
\|_{\infty\to\infty}
&=\|(Df(x_0)-S+S)\circ T-\operatorname{id}_{\mathbb{R}^n}
\|_{\infty\to\infty}\\
&\leq\|S-Df(x_0)\|_{\infty\to\infty}\|T\|_{\infty\to\infty}+\|S\circ T-\operatorname{id}_{\mathbb{R}^n}
\|_{\infty\to\infty},
\end{align*}
it suffices to prove the estimate
\[
\|S-Df(x_0)\|_{\infty\to\infty}
\leq \tfrac{1}{2}\tilde\delta|f|d(d-1).
\]
To this end, we apply the following version of Taylor's theorem:
Given a compact interval $I$ with midpoint $a$ and a smooth function $p\colon I\to\mathbb{R}$, then for every $x\in I$, it holds that
\[
\big| p(x) - \big( p(a) + p'(a)(x-a)\big) \big|
\leq \tfrac{1}{2}|x-a|^2\cdot \sup_{t\in I}|p''(t)|.
\]
Letting $S_{ij}$ denote the $(i,j)$-entry of the matrix representation of $S$, then taking $p\colon[-\delta,\delta]\to\mathbb{R}$ defined by $p(t)=f_i(x_0+te_j)$  and applying Taylor's theorem gives
\[
\Big|S_{ij}-\tfrac{\partial f_i(x)}{\partial x_j}\big|_{x=x_0}\Big|
=\tfrac{1}{\delta}\big|
p(\delta) -\big(
p(0) + p'(0) \delta 
\big)
\big|
\leq\tfrac{\delta}{2}\sup_{t\in[-\delta,\delta]}|p''(t)|
\leq \tfrac{\delta}{2}\sup_{\|z\|_\infty\leq\|x_0\|_\infty+\delta}\Big|\tfrac{\partial^2 f_i(x)}{\partial x_j^2}\big|_{x=z}\Big|.
\]
Next, writing $f_i(x)=\sum_{\alpha\in E_i}c_\alpha^{(i)}x^\alpha$ for some set $E_i$ of multi-indices $\alpha$ with $\sum_j\alpha_j\leq d$, then the second partial derivative is $\tfrac{\partial^2 f_i(x)}{\partial x_j^2}=\sum_{\alpha\in E_i}c_\alpha^{(i)}\alpha_j(\alpha_j-1)x^{\alpha-2e_j}$, and so
\[
\Big|\tfrac{\partial^2 f_i(x)}{\partial x_j^2}\big|_{x=z}\Big|
\leq \sum_{\alpha\in E_i}|c_\alpha^{(i)}|\cdot\alpha_j(\alpha_j-1)\cdot|z^{\alpha-2e_j}|.
\]
Putting everything together, we obtain
\begin{align*}
\|S-Df(x_0)\|_{\infty\to\infty}
&=\max_{i\in[n]}\sum_{j\in[m]}\Big|S_{ij}-\tfrac{\partial f_i(x)}{\partial x_j}\big|_{x=x_0}\Big|\\
&\leq\max_{i\in[n]}\sum_{j\in[m]}\tfrac{\delta}{2}\sup_{\|z\|_\infty\leq\|x_0\|_\infty+\delta}\Big|\tfrac{\partial^2 f_i(x)}{\partial x_j^2}\big|_{x=z}\Big|\\
&\leq\max_{i\in[n]}\sum_{j\in[m]}\tfrac{\delta}{2}\sup_{\|z\|_\infty\leq\|x_0\|_\infty+\delta}\sum_{\alpha\in E_i}|c_\alpha^{(i)}|\cdot\alpha_j(\alpha_j-1)\cdot|z^{\alpha-2e_j}|\\
&\leq\tfrac{\delta}{2}\max_{i\in[n]}\sum_{\alpha\in E_i}|c_\alpha^{(i)}|\sum_{j\in[m]}\alpha_j(\alpha_j-1)\sup_{\|z\|_\infty\leq\|x_0\|_\infty+\delta}|z^{\alpha-2e_j}|.
\end{align*}
To estimate the supremum, observe that $\|z\|_\infty\leq\|x_0\|_\infty+\delta$ implies
\[
|z^{\alpha-2e_j}|
=\prod_{k\in[m]}|z_k|^{\alpha_k-2\delta_{jk}}
\leq \prod_{k\in[m]}\max\{1,\|x_0\|_\infty+\delta\}^{\alpha_k-2\delta_{jk}}
\leq \max\{1,\|x_0\|_\infty+\delta\}^{d-2}.
\]
Next, each $\alpha\in E_i$ satisfies $\alpha_j\geq0$ for every $j$ and $\sum_j\alpha_j\leq d$, and so
\[
\sum_j\alpha_j(\alpha_j-1)
=\sum_j\alpha_j^2-\sum_j\alpha_j
\leq \sum_j\alpha_j^2+\sum_i\sum_{j\neq i}\alpha_i\alpha_j-\sum_j\alpha_j
=\bigg(\sum_j\alpha_j\bigg)\bigg(\sum_j\alpha_j-1\bigg)
\leq d(d-1).
\]
Finally, we recall that $|f_i|:=\sum_{\alpha\in E_i}|c_\alpha^{(i)}|$ and $|f|:=\max_{i\in[n]}|f_i|$.
The result follows.
\end{proof}

\begin{proof}[Proof sketch of Theorem~\ref{thm.strong d by 2d for small d}]
First, we describe a polynomial map $f\colon\mathbb{R}^{4d+1}\to\mathbb{R}^{2d+\lfloor\frac{d}{2}\rfloor+1}$ such that 
\[
f(\operatorname{Re}x,\operatorname{Im}x,\operatorname{Re}y,\operatorname{Im}y,w)
=0
\]
precisely when $x,y\in\mathbb{C}^d$ generate a $2$-circulant $d\times 2d$ ETF and $w=\frac{1}{2}$.
To do this, we follow Lemma~\ref{lem.2-circ constraints}.
The first two coordinate functions of $f$ force $\|x\|^2-1=0$ and $\|y\|^2-1=0$ so as to implement \eqref{eq.2circ unit norm}.
Next, we implement \eqref{eq.2circ tight} in the time domain by forcing
\[
\langle x,T^j x\rangle + \langle y,T^j y\rangle
-4w\delta_{j,0}
=0
\qquad
\forall j\in\{0,\ldots,d-1\}.
\]
When $j=0$, this is a real constraint.
When $j=\frac{d}{2}$ with $d$ even, this is also a real constraint.
For $0<j<\frac{d}{2}$, we take the real and imaginary parts of this constraint, which in turn imply that the constraint holds with $d-j$ in place of $j$.
Overall, \eqref{eq.2circ tight} accounts for $d$ coordinate functions of $f$.
We implement \eqref{eq.2circ equi 1} and \eqref{eq.2circ equi 2} with the remaining coordinate functions:
\begin{align*}
|\langle x,T^jx\rangle|^2-|\langle x,y\rangle|^2=0&
\qquad \forall j\in\{1,\ldots,\lfloor\tfrac{d}{2}\rfloor\},\\
|\langle x,T^jy\rangle|^2-|\langle x,y\rangle|^2=0&
\qquad \forall j\in\{1,\ldots,d-1\}.
\end{align*}
The resulting polynomial map $f$ has total degree $4$ and $|f|\leq 16d^2$.
Take any $z_0\in\mathbb{R}^{4d+1}$ whose first $4d$ coordinates determine a numerical $2$-circulant ETF, and whose last coordinate is $\frac{1}{2}$, thereby implying $f(z_0)\approx 0$.
Put $\delta:=10^{-10}$, compute the $\delta$-secant approximation $S$ of the Jacobian of $f$ at $z_0$, and then numerically compute the pseudoinverse $T$ of $S$.
It remains to test whether there exists $\varepsilon>0$ that satisfies the inequality in Theorem~\ref{thm.secant jacobian}.
Since the $w$-coordinate of $z_0$ is $\frac{1}{2}$, it turns out that $\|z_0\|_\infty<1$ in practice, and so we may insist on $\tilde\varepsilon=\varepsilon$, in which case the inequality reduces to a quadratic in $\varepsilon$.
As such, whether there exists an appropriate $\varepsilon>0$ amounts to testing a few inequalities in the coefficients of this quadratic, which we implement using interval arithmetic in Mathematica.
We provide our code as an ancillary file with the arXiv version of this paper.
\end{proof}

\section{Discussion}
\label{sec.discussion}

In this paper, we reviewed what is known about $d\times2d$ ETFs, we introduced a new infinite family of $d\times2d$ ETFs, we observed their emergent $2$-circulant structure, we used this structure to formulate strengthened versions of the $d\times 2d$ conjecture from~\cite{FallonI:23}, and we made partial progress on both strengthenings.
In this section, we reflect on our results and discuss opportunities for future work.

The current state of play for $d\leq150$ is summarized in Table~\ref{table.first 150}.
Circles denote problems that were open before this paper, and checkmarks denote problems that are now solved.
The column labeled ``$\exists$'' indicates whether a $d\times 2d$ ETF is known to exist.
Every entry in this column has a checkmark as a consequence of Theorem~\ref{thm.strong d by 2d for small d}.
The column labeled ``x'' indicates whether an explicit $d\times 2d$ ETF is known.
At the moment, the only $d\leq 150$ for which no explicit $d\times 2d$ ETF is known are $d\in\{77,93,105,133\}$.
(Curiously, dimensions $77$, $93$, and $133$ all satisfy the integrality conditions that are necessary for a \textit{real} ETF to exist~\cite{FickusM:15}.)
The column labeled ``$\exists2$'' indicates whether a $2$-circulant $d\times 2d$ ETF is known to exist, and the column labeled ``x$2$'' indicates whether an \textit{explicit} $2$-circulant $d\times 2d$ ETF is known.
Every entry in the $\exists2$ column has a checkmark as a consequence of Theorem~\ref{thm.strong d by 2d for small d}.
Almost every checkmark in the x$2$ column follows from Theorem~\ref{thm.2-circ ETF families}, and the smallest open case here is $d=33$.
Explicit constructions are listed in the ``notes'' column, with bold entries belonging to the new infinite family in Theorem~\ref{thm.doubling signature}.
We do not describe all known constructions in the notes column.
In particular, if an explicit $2$-circulant ETF is known, then we only list such constructions.
While Theorem~\ref{thm.2-circ ETF families} is our primary source of explicit $2$-circulant ETFs, other constructions also appear in Table~\ref{table.first 150}. 
For example, when $p$ and $p+2$ are both prime, there exists a difference set $\operatorname{TPP}(p)$ of size $d:=\frac{p(p+2)-1}{2}$ in the cyclic group $C_p\times C_{p+2}$, which in turn induces a harmonic ETF of size $d\times (2d+1)$; see~\cite{JungnickelPS:07,FickusM:15}.
Applying Proposition~\ref{prop.etf doubling} then doubles this ETF to produce a $p(p+2)\times 2p(p+2)$ ETF that we label by $2\cdot\operatorname{TPP}(p)$.
In general, the double of any harmonic ETF over a cyclic group is $2$-circulant by Theorem~\ref{thm.t-gen Gamma-harm detector}.
For this reason, $2\cdot\operatorname{TPP}(p)$ appears in Table~\ref{thm.doubling signature} for every $p\in\{3,5,11\}$.
Similarly, \cite{Gordon:online} reports a difference set of size $31$ in $C_{63}$ called $\operatorname{Singer}(5,2)$, which we double to produce a $2$-circulant $63\times 126$ ETF.
When an explicit $2$-circulant construction is not known, we list constructions of the form $F_{v}+1$ or $2\cdot F_v$.
The conference graph $F_{65}$ was recently discovered in~\cite{Gritsenko:21}.
All other non-Paley $F_v$'s that appear in Table~\ref{table.first 150} satisfy $v=4k-1$ for some $k$, in which case $F_v$ is known to exist whenever $k<89$; see~\cite{Djokovic:23} and references therein.

For future work, we describe a few possible approaches to the $d\times 2d$ conjecture.
Is there any hope of finding an explicit solution to the weak form of the conjecture?
Perhaps one may impose additional structure beyond $2$-circulant to obtain a discrete solution set that is more amenable to tools from algebraic combinatorics.
Note that the skew Hadamard conjecture implies ``three-fourths'' of the $d\times 2d$ conjecture in the sense that it produces $d\times 2d$ ETFs for every $d\not\equiv 1\bmod 4$.
(Explicitly, take $F_{2d-1}+1$ when $d$ is even and $2\cdot F_{d}$ when $d\equiv 3\bmod 4$.)
Is there a conditional proof of the remaining case where $d\equiv 1\bmod 4$?
In lieu of an explicit solution, one might seek an implicit proof of the $d\times 2d$ conjecture.
For example, are the defining polynomials amenable to tools from (real) algebraic geometry?
In practice, we find that numerical $d\times 2d$ ETFs are easy to obtain by local methods.
In particular, generic initializations tend to avoid spurious local minimizers (if they even exist).
Can one prove the $d\times 2d$ conjecture by performing a landscape analysis akin to~\cite{BenedettoF:03,MixonNSV:24}?

\begin{table}[h]
\caption{Parameters of known $d\times 2d$ ETFs (see Section~\ref{sec.discussion} for details) \label{table.first 150}}
\begin{center}
\scriptsize{
\begin{tabular}{rlllll}
$d$ & $\exists$ & x & \!$\exists2$ & \!x$2$ & notes \\ \hline
$1$ & \checkmark & \checkmark & \checkmark & \checkmark & trivial \\
$2$ & \checkmark & \checkmark & \checkmark & \checkmark & $G_3+1$ \\
$3$ & \checkmark & \checkmark & \checkmark & \checkmark & $G_5+1$, $2\cdot G_3$ \\
$4$ & \checkmark & \checkmark & \checkmark\hspace{-1em}$\bigcirc$ & \checkmark\hspace{-1em}$\bigcirc$ & $G_7+1$, $2\cdot(G_3+1)$ \\
$5$ & \checkmark & \checkmark & \checkmark\hspace{-1em}$\bigcirc$ & \checkmark\hspace{-1em}$\bigcirc$ & $G_9+1$, $\boldsymbol{2\cdot G_5}$ \\
$6$ & \checkmark & \checkmark & \checkmark\hspace{-1em}$\bigcirc$ & \checkmark\hspace{-1em}$\bigcirc$ & $G_{11}+1$, $2\cdot (G_5+1)$ \\
$7$ & \checkmark & \checkmark & \checkmark\hspace{-1em}$\bigcirc$ & \checkmark\hspace{-1em}$\bigcirc$ & $G_{13}+1$, $2\cdot G_7$ \\
$8$ & \checkmark & \checkmark & \checkmark\hspace{-1em}$\bigcirc$ & \checkmark\hspace{-1em}$\bigcirc$ & $2\cdot(G_7+1)$ \\
$9$ & \checkmark & \checkmark & \checkmark\hspace{-1em}$\bigcirc$ & \checkmark\hspace{-1em}$\bigcirc$ & $G_{17}+1$ \\
$10$ & \checkmark & \checkmark & \checkmark\hspace{-1em}$\bigcirc$ & \checkmark\hspace{-1em}$\bigcirc$ & $G_{19}+1$, $2\cdot (G_{9}+1)$ \\
$11$ & \checkmark & \checkmark & \checkmark\hspace{-1em}$\bigcirc$ & \checkmark\hspace{-1em}$\bigcirc$ & $2\cdot G_{11}$ \\
$12$ & \checkmark & \checkmark & \checkmark\hspace{-1em}$\bigcirc$ & \checkmark\hspace{-1em}$\bigcirc$ & $G_{23}+1$, $2\cdot (G_{11}+1)$ \\
$13$ & \checkmark & \checkmark & \checkmark\hspace{-1em}$\bigcirc$ & \checkmark\hspace{-1em}$\bigcirc$ & $G_{25}+1$, $\boldsymbol{2\cdot G_{13}}$ \\
$14$ & \checkmark & \checkmark & \checkmark\hspace{-1em}$\bigcirc$ & \checkmark\hspace{-1em}$\bigcirc$ & $2\cdot (G_{13}+1)$ \\
$15$ & \checkmark & \checkmark & \checkmark\hspace{-1em}$\bigcirc$ & \checkmark\hspace{-1em}$\bigcirc$ & $G_{29}+1$, $2\cdot\operatorname{TPP}(3)$ \\
$16$ & \checkmark & \checkmark & \checkmark\hspace{-1em}$\bigcirc$ & \checkmark\hspace{-1em}$\bigcirc$ & $G_{31}+1$ \\
$17$ & \checkmark\hspace{-1em}$\bigcirc$ & \checkmark\hspace{-1em}$\bigcirc$ & \checkmark\hspace{-1em}$\bigcirc$ & \checkmark\hspace{-1em}$\bigcirc$ & $\boldsymbol{2\cdot G_{17}}$ \\
$18$ & \checkmark & \checkmark & \checkmark\hspace{-1em}$\bigcirc$ & \checkmark\hspace{-1em}$\bigcirc$ & $2\cdot (G_{17}+1)$ \\
$19$ & \checkmark & \checkmark & \checkmark\hspace{-1em}$\bigcirc$ & \checkmark\hspace{-1em}$\bigcirc$ & $G_{37}+1$, $2\cdot G_{19}$ \\
$20$ & \checkmark & \checkmark & \checkmark\hspace{-1em}$\bigcirc$ & \checkmark\hspace{-1em}$\bigcirc$ & $2\cdot (G_{19}+1)$ \\
$21$ & \checkmark & \checkmark & \checkmark\hspace{-1em}$\bigcirc$ & \checkmark\hspace{-1em}$\bigcirc$ & $G_{41}+1$ \\
$22$ & \checkmark & \checkmark & \checkmark\hspace{-1em}$\bigcirc$ & \checkmark\hspace{-1em}$\bigcirc$ & $G_{43}+1$ \\
$23$ & \checkmark & \checkmark & \checkmark\hspace{-1em}$\bigcirc$ & \checkmark\hspace{-1em}$\bigcirc$ & $2\cdot G_{23}$ \\
$24$ & \checkmark & \checkmark & \checkmark\hspace{-1em}$\bigcirc$ & \checkmark\hspace{-1em}$\bigcirc$ & $G_{47}+1$, $2\cdot (G_{23}+1)$ \\
$25$ & \checkmark & \checkmark & \checkmark\hspace{-1em}$\bigcirc$ & \checkmark\hspace{-1em}$\bigcirc$ & $G_{49}+1$ \\
$26$ & \checkmark & \checkmark & \checkmark\hspace{-1em}$\bigcirc$ & \checkmark\hspace{-1em}$\bigcirc$ & $2\cdot (G_{25}+1)$ \\
$27$ & \checkmark & \checkmark & \checkmark\hspace{-1em}$\bigcirc$ & \checkmark\hspace{-1em}$\bigcirc$ & $G_{53}+1$, $2\cdot G_{27}$ \\
$28$ & \checkmark & \checkmark & \checkmark\hspace{-1em}$\bigcirc$ & \checkmark\hspace{-1em}$\bigcirc$ & $2\cdot (G_{27}+1)$ \\
$29$ & \checkmark\hspace{-1em}$\bigcirc$ & \checkmark\hspace{-1em}$\bigcirc$ & \checkmark\hspace{-1em}$\bigcirc$ & \checkmark\hspace{-1em}$\bigcirc$ & $\boldsymbol{2\cdot G_{29}}$ \\
$30$ & \checkmark & \checkmark & \checkmark\hspace{-1em}$\bigcirc$ & \checkmark\hspace{-1em}$\bigcirc$ & $G_{59}+1$, $2\cdot (G_{29}+1)$ \\
$31$ & \checkmark & \checkmark & \checkmark\hspace{-1em}$\bigcirc$ & \checkmark\hspace{-1em}$\bigcirc$ & $G_{61}+1$, $2\cdot G_{31}$ \\
$32$ & \checkmark & \checkmark & \checkmark\hspace{-1em}$\bigcirc$ & \checkmark\hspace{-1em}$\bigcirc$ & $2\cdot (G_{31}+1)$ \\
$33$ & \checkmark & \checkmark & \checkmark\hspace{-1em}$\bigcirc$ & \phantom{\checkmark}\hspace{-1em}$\bigcirc$ & $F_{65}+1$ \\
$34$ & \checkmark & \checkmark & \checkmark\hspace{-1em}$\bigcirc$ & \checkmark\hspace{-1em}$\bigcirc$ & $G_{67}+1$ \\
$35$ & \checkmark & \checkmark & \checkmark\hspace{-1em}$\bigcirc$ & \checkmark\hspace{-1em}$\bigcirc$ & $2\cdot\operatorname{TPP}(5)$ \\
$36$ & \checkmark & \checkmark & \checkmark\hspace{-1em}$\bigcirc$ & \checkmark\hspace{-1em}$\bigcirc$ & $G_{71}+1$ \\
$37$ & \checkmark & \checkmark & \checkmark\hspace{-1em}$\bigcirc$ & \checkmark\hspace{-1em}$\bigcirc$ & $G_{73}+1$, $\boldsymbol{2\cdot G_{37}}$ \\
$38$ & \checkmark & \checkmark & \checkmark\hspace{-1em}$\bigcirc$ & \checkmark\hspace{-1em}$\bigcirc$ & $2\cdot (G_{37}+1)$ \\
$39$ & \checkmark & \checkmark & \checkmark\hspace{-1em}$\bigcirc$ & \phantom{\checkmark}\hspace{-1em}$\bigcirc$ & $2\cdot F_{39}$ \\
$40$ & \checkmark & \checkmark & \checkmark\hspace{-1em}$\bigcirc$ & \checkmark\hspace{-1em}$\bigcirc$ & $G_{79}+1$ \\
$41$ & \checkmark & \checkmark & \checkmark\hspace{-1em}$\bigcirc$ & \checkmark\hspace{-1em}$\bigcirc$ & $\boldsymbol{2\cdot G_{41}}$ \\
$42$ & \checkmark & \checkmark & \checkmark\hspace{-1em}$\bigcirc$ & \checkmark\hspace{-1em}$\bigcirc$ & $G_{83}+1$, $2\cdot (G_{41}+1)$ \\
$43$ & \checkmark & \checkmark & \checkmark\hspace{-1em}$\bigcirc$ & \checkmark\hspace{-1em}$\bigcirc$ & $2\cdot G_{43}$ \\
$44$ & \checkmark & \checkmark & \checkmark\hspace{-1em}$\bigcirc$ & \checkmark\hspace{-1em}$\bigcirc$ & $2\cdot (G_{43}+1)$ \\
$45$ & \checkmark & \checkmark & \checkmark\hspace{-1em}$\bigcirc$ & \checkmark\hspace{-1em}$\bigcirc$ & $G_{89}+1$ \\
$46$ & \checkmark & \checkmark & \checkmark\hspace{-1em}$\bigcirc$ & \phantom{\checkmark}\hspace{-1em}$\bigcirc$ & $F_{91}+1$ \\
$47$ & \checkmark & \checkmark & \checkmark\hspace{-1em}$\bigcirc$ & \checkmark\hspace{-1em}$\bigcirc$ & $2\cdot G_{47}$ \\
$48$ & \checkmark & \checkmark & \checkmark\hspace{-1em}$\bigcirc$ & \checkmark\hspace{-1em}$\bigcirc$ & $2\cdot (G_{47}+1)$ \\
$49$ & \checkmark & \checkmark & \checkmark\hspace{-1em}$\bigcirc$ & \checkmark\hspace{-1em}$\bigcirc$ & $G_{97}+1$ \\
$50$ & \checkmark & \checkmark & \checkmark\hspace{-1em}$\bigcirc$ & \checkmark\hspace{-1em}$\bigcirc$ & $2\cdot (G_{49}+1)$ \\
$51$ & \checkmark & \checkmark & \checkmark\hspace{-1em}$\bigcirc$ & \checkmark\hspace{-1em}$\bigcirc$ & $G_{101}+1$ \\
$52$ & \checkmark & \checkmark & \checkmark\hspace{-1em}$\bigcirc$ & \checkmark\hspace{-1em}$\bigcirc$ & $G_{103}+1$ \\
$53$ & \checkmark\hspace{-1em}$\bigcirc$ & \checkmark\hspace{-1em}$\bigcirc$ & \checkmark\hspace{-1em}$\bigcirc$ & \checkmark\hspace{-1em}$\bigcirc$ &  $\boldsymbol{2\cdot G_{53}}$ \\
$54$ & \checkmark & \checkmark & \checkmark\hspace{-1em}$\bigcirc$ & \checkmark\hspace{-1em}$\bigcirc$ & $G_{107}+1$, $2\cdot (G_{53}+1)$ \\
$55$ & \checkmark & \checkmark & \checkmark\hspace{-1em}$\bigcirc$ & \checkmark\hspace{-1em}$\bigcirc$ & $G_{109}+1$ \\
$56$ & \checkmark & \checkmark & \checkmark\hspace{-1em}$\bigcirc$ & \phantom{\checkmark}\hspace{-1em}$\bigcirc$ & $F_{111}+1$ \\
$57$ & \checkmark & \checkmark & \checkmark\hspace{-1em}$\bigcirc$ & \checkmark\hspace{-1em}$\bigcirc$ & $G_{113}+1$ \\
$58$ & \checkmark & \checkmark & \checkmark\hspace{-1em}$\bigcirc$ & \phantom{\checkmark}\hspace{-1em}$\bigcirc$ & $F_{115}+1$ \\
$59$ & \checkmark & \checkmark & \checkmark\hspace{-1em}$\bigcirc$ & \checkmark\hspace{-1em}$\bigcirc$ & $2\cdot G_{59}$ \\
$60$ & \checkmark & \checkmark & \checkmark\hspace{-1em}$\bigcirc$ & \checkmark\hspace{-1em}$\bigcirc$ & $2\cdot (G_{59}+1)$ \\
$61$ & \checkmark & \checkmark & \checkmark\hspace{-1em}$\bigcirc$ & \checkmark\hspace{-1em}$\bigcirc$ & $G_{121}+1$, $\boldsymbol{2\cdot G_{61}}$ \\
$62$ & \checkmark & \checkmark & \checkmark\hspace{-1em}$\bigcirc$ & \checkmark\hspace{-1em}$\bigcirc$ & $2\cdot (G_{61}+1)$ \\
$63$ & \checkmark & \checkmark & \checkmark\hspace{-1em}$\bigcirc$ & \checkmark\hspace{-1em}$\bigcirc$ & $2\cdot\operatorname{Singer}(5,2)$ \\
$64$ & \checkmark & \checkmark & \checkmark\hspace{-1em}$\bigcirc$ & \checkmark\hspace{-1em}$\bigcirc$ & $G_{127}+1$ \\
$65$ & \checkmark\hspace{-1em}$\bigcirc$ & \checkmark\hspace{-1em}$\bigcirc$ & \checkmark\hspace{-1em}$\bigcirc$ & \phantom{\checkmark}\hspace{-1em}$\bigcirc$ & $\boldsymbol{2\cdot F_{65}}$ \\
$66$ & \checkmark & \checkmark & \checkmark\hspace{-1em}$\bigcirc$ & \checkmark\hspace{-1em}$\bigcirc$ & $G_{131}+1$ \\
$67$ & \checkmark & \checkmark & \checkmark\hspace{-1em}$\bigcirc$ & \checkmark\hspace{-1em}$\bigcirc$ & $2\cdot G_{67}$ \\
$68$ & \checkmark & \checkmark & \checkmark\hspace{-1em}$\bigcirc$ & \checkmark\hspace{-1em}$\bigcirc$ & $2\cdot (G_{67}+1)$ \\
$69$ & \checkmark & \checkmark & \checkmark\hspace{-1em}$\bigcirc$ & \checkmark\hspace{-1em}$\bigcirc$ & $G_{137}+1$ \\
$70$ & \checkmark & \checkmark & \checkmark\hspace{-1em}$\bigcirc$ & \checkmark\hspace{-1em}$\bigcirc$ & $G_{139}+1$ \\
$71$ & \checkmark & \checkmark & \checkmark\hspace{-1em}$\bigcirc$ & \checkmark\hspace{-1em}$\bigcirc$ & $2\cdot G_{71}$ \\
$72$ & \checkmark & \checkmark & \checkmark\hspace{-1em}$\bigcirc$ & \checkmark\hspace{-1em}$\bigcirc$ & $2\cdot (G_{71}+1)$ \\
$73$ & \checkmark\hspace{-1em}$\bigcirc$ & \checkmark\hspace{-1em}$\bigcirc$ & \checkmark\hspace{-1em}$\bigcirc$ & \checkmark\hspace{-1em}$\bigcirc$ & $\boldsymbol{2\cdot G_{73}}$ \\
$74$ & \checkmark & \checkmark & \checkmark\hspace{-1em}$\bigcirc$ & \checkmark\hspace{-1em}$\bigcirc$ & $2\cdot (G_{73}+1)$ \\
$75$ & \checkmark & \checkmark & \checkmark\hspace{-1em}$\bigcirc$ & \checkmark\hspace{-1em}$\bigcirc$ & $G_{149}+1$
\end{tabular}
}
\qquad\qquad
\scriptsize{
\begin{tabular}{rlllll}
$d$ & $\exists$ & x & \!$\exists2$ & \!x$2$ & notes \\ \hline
$76$ & \checkmark & \checkmark & \checkmark\hspace{-1em}$\bigcirc$ & \checkmark\hspace{-1em}$\bigcirc$ & $G_{151}+1$ \\
$77$ & \checkmark\hspace{-1em}$\bigcirc$ & \phantom{\checkmark}\hspace{-1em}$\bigcirc$ & \checkmark\hspace{-1em}$\bigcirc$ & \phantom{\checkmark}\hspace{-1em}$\bigcirc$ & \\
$78$ & \checkmark & \checkmark & \checkmark\hspace{-1em}$\bigcirc$ & \phantom{\checkmark}\hspace{-1em}$\bigcirc$ & $F_{155}+1$ \\
$79$ & \checkmark & \checkmark & \checkmark\hspace{-1em}$\bigcirc$ & \checkmark\hspace{-1em}$\bigcirc$ & $G_{157}+1$, $2\cdot G_{79}$ \\
$80$ & \checkmark & \checkmark & \checkmark\hspace{-1em}$\bigcirc$ & \checkmark\hspace{-1em}$\bigcirc$ & $2\cdot (G_{79}+1)$ \\
$81$ & \checkmark\hspace{-1em}$\bigcirc$ & \checkmark\hspace{-1em}$\bigcirc$ & \checkmark\hspace{-1em}$\bigcirc$ & \phantom{\checkmark}\hspace{-1em}$\bigcirc$ & $\boldsymbol{2\cdot G_{81}}$ \\
$82$ & \checkmark & \checkmark & \checkmark\hspace{-1em}$\bigcirc$ & \checkmark\hspace{-1em}$\bigcirc$ & $G_{163}+1$, $2\cdot (G_{81}+1)$ \\
$83$ & \checkmark & \checkmark & \checkmark\hspace{-1em}$\bigcirc$ & \checkmark\hspace{-1em}$\bigcirc$ & $2\cdot G_{83}$ \\
$84$ & \checkmark & \checkmark & \checkmark\hspace{-1em}$\bigcirc$ & \checkmark\hspace{-1em}$\bigcirc$ & $G_{167}+1$, $2\cdot (G_{83}+1)$ \\
$85$ & \checkmark & \checkmark & \checkmark\hspace{-1em}$\bigcirc$ & \checkmark\hspace{-1em}$\bigcirc$ & $G_{169}+1$ \\
$86$ & \checkmark & \checkmark & \checkmark\hspace{-1em}$\bigcirc$ & \phantom{\checkmark}\hspace{-1em}$\bigcirc$ & $F_{171}+1$ \\
$87$ & \checkmark & \checkmark & \checkmark\hspace{-1em}$\bigcirc$ & \checkmark\hspace{-1em}$\bigcirc$ & $G_{173}+1$ \\
$88$ & \checkmark & \checkmark & \checkmark\hspace{-1em}$\bigcirc$ & \phantom{\checkmark}\hspace{-1em}$\bigcirc$ & $F_{175}+1$ \\
$89$ & \checkmark\hspace{-1em}$\bigcirc$ & \checkmark\hspace{-1em}$\bigcirc$ & \checkmark\hspace{-1em}$\bigcirc$ & \checkmark\hspace{-1em}$\bigcirc$ & $\boldsymbol{2\cdot G_{89}}$ \\
$90$ & \checkmark & \checkmark & \checkmark\hspace{-1em}$\bigcirc$ & \checkmark\hspace{-1em}$\bigcirc$ & $G_{179}+1$, $2\cdot (G_{89}+1)$ \\
$91$ & \checkmark & \checkmark & \checkmark\hspace{-1em}$\bigcirc$ & \checkmark\hspace{-1em}$\bigcirc$ & $G_{181}+1$ \\
$92$ & \checkmark & \checkmark & \checkmark\hspace{-1em}$\bigcirc$ & \phantom{\checkmark}\hspace{-1em}$\bigcirc$ & $F_{183}+1$ \\
$93$ & \checkmark\hspace{-1em}$\bigcirc$ & \phantom{\checkmark}\hspace{-1em}$\bigcirc$ & \checkmark\hspace{-1em}$\bigcirc$ & \phantom{\checkmark}\hspace{-1em}$\bigcirc$ & \\
$94$ & \checkmark & \checkmark & \checkmark\hspace{-1em}$\bigcirc$ & \phantom{\checkmark}\hspace{-1em}$\bigcirc$ & $F_{187}+1$ \\
$95$ & \checkmark & \checkmark & \checkmark\hspace{-1em}$\bigcirc$ & \phantom{\checkmark}\hspace{-1em}$\bigcirc$ & $2\cdot F_{95}$ \\
$96$ & \checkmark & \checkmark & \checkmark\hspace{-1em}$\bigcirc$ & \checkmark\hspace{-1em}$\bigcirc$ & $G_{191}+1$ \\
$97$ & \checkmark & \checkmark & \checkmark\hspace{-1em}$\bigcirc$ & \checkmark\hspace{-1em}$\bigcirc$ & $G_{193}+1$, $\boldsymbol{2\cdot G_{97}}$ \\
$98$ & \checkmark & \checkmark & \checkmark\hspace{-1em}$\bigcirc$ & \checkmark\hspace{-1em}$\bigcirc$ & $2\cdot (G_{97}+1)$ \\
$99$ & \checkmark & \checkmark & \checkmark\hspace{-1em}$\bigcirc$ & \checkmark\hspace{-1em}$\bigcirc$ & $G_{197}+1$ \\
$100$ & \checkmark & \checkmark & \checkmark\hspace{-1em}$\bigcirc$ & \checkmark\hspace{-1em}$\bigcirc$ & $G_{199}+1$ \\
$101$ & \checkmark\hspace{-1em}$\bigcirc$ & \checkmark\hspace{-1em}$\bigcirc$ & \checkmark\hspace{-1em}$\bigcirc$ & \checkmark\hspace{-1em}$\bigcirc$ & $\boldsymbol{2\cdot G_{101}}$ \\
$102$ & \checkmark & \checkmark & \checkmark\hspace{-1em}$\bigcirc$ & \checkmark\hspace{-1em}$\bigcirc$ & $2\cdot (G_{101}+1)$ \\
$103$ & \checkmark & \checkmark & \checkmark\hspace{-1em}$\bigcirc$ & \checkmark\hspace{-1em}$\bigcirc$ & $2\cdot G_{103}$ \\
$104$ & \checkmark & \checkmark & \checkmark\hspace{-1em}$\bigcirc$ & \checkmark\hspace{-1em}$\bigcirc$ & $2\cdot (G_{103}+1)$ \\
$105$ & \checkmark\hspace{-1em}$\bigcirc$ & \phantom{\checkmark}\hspace{-1em}$\bigcirc$ & \checkmark\hspace{-1em}$\bigcirc$ & \phantom{\checkmark}\hspace{-1em}$\bigcirc$ & \\
$106$ & \checkmark & \checkmark & \checkmark\hspace{-1em}$\bigcirc$ & \checkmark\hspace{-1em}$\bigcirc$ & $G_{211}+1$ \\
$107$ & \checkmark & \checkmark & \checkmark\hspace{-1em}$\bigcirc$ & \checkmark\hspace{-1em}$\bigcirc$ & $2\cdot G_{107}$ \\
$108$ & \checkmark & \checkmark & \checkmark\hspace{-1em}$\bigcirc$ & \checkmark\hspace{-1em}$\bigcirc$ & $2\cdot (G_{107}+1)$ \\
$109$ & \checkmark\hspace{-1em}$\bigcirc$ & \checkmark\hspace{-1em}$\bigcirc$ & \checkmark\hspace{-1em}$\bigcirc$ & \checkmark\hspace{-1em}$\bigcirc$ & $\boldsymbol{2\cdot G_{109}}$ \\
$110$ & \checkmark & \checkmark & \checkmark\hspace{-1em}$\bigcirc$ & \checkmark\hspace{-1em}$\bigcirc$ & $2\cdot (G_{109}+1)$ \\
$111$ & \checkmark & \checkmark & \checkmark\hspace{-1em}$\bigcirc$ & \phantom{\checkmark}\hspace{-1em}$\bigcirc$ & $2\cdot F_{111}$ \\
$112$ & \checkmark & \checkmark & \checkmark\hspace{-1em}$\bigcirc$ & \checkmark\hspace{-1em}$\bigcirc$ & $G_{223}+1$ \\
$113$ & \checkmark & \checkmark & \checkmark\hspace{-1em}$\bigcirc$ & \checkmark\hspace{-1em}$\bigcirc$ & $\boldsymbol{2\cdot G_{113}}$ \\
$114$ & \checkmark & \checkmark & \checkmark\hspace{-1em}$\bigcirc$ & \checkmark\hspace{-1em}$\bigcirc$ & $G_{227}+1$, $2\cdot (G_{113}+1)$ \\
$115$ & \checkmark & \checkmark & \checkmark\hspace{-1em}$\bigcirc$ & \checkmark\hspace{-1em}$\bigcirc$ & $G_{229}+1$ \\
$116$ & \checkmark & \checkmark & \checkmark\hspace{-1em}$\bigcirc$ & \phantom{\checkmark}\hspace{-1em}$\bigcirc$ & $F_{231}+1$ \\
$117$ & \checkmark & \checkmark & \checkmark\hspace{-1em}$\bigcirc$ & \checkmark\hspace{-1em}$\bigcirc$ & $G_{233}+1$ \\
$118$ & \checkmark & \checkmark & \checkmark\hspace{-1em}$\bigcirc$ & \phantom{\checkmark}\hspace{-1em}$\bigcirc$ & $F_{235}+1$ \\
$119$ & \checkmark & \checkmark & \checkmark\hspace{-1em}$\bigcirc$ & \phantom{\checkmark}\hspace{-1em}$\bigcirc$ & $2\cdot F_{119}$ \\
$120$ & \checkmark & \checkmark & \checkmark\hspace{-1em}$\bigcirc$ & \checkmark\hspace{-1em}$\bigcirc$ & $G_{239}+1$ \\
$121$ & \checkmark & \checkmark & \checkmark\hspace{-1em}$\bigcirc$ & \checkmark\hspace{-1em}$\bigcirc$ & $G_{241}+1$ \\
$122$ & \checkmark & \checkmark & \checkmark\hspace{-1em}$\bigcirc$ & \checkmark\hspace{-1em}$\bigcirc$ & $2\cdot (G_{121}+1)$ \\
$123$ & \checkmark & \checkmark & \checkmark\hspace{-1em}$\bigcirc$ & \phantom{\checkmark}\hspace{-1em}$\bigcirc$ & $2\cdot F_{123}$ \\
$124$ & \checkmark & \checkmark & \checkmark\hspace{-1em}$\bigcirc$ & \phantom{\checkmark}\hspace{-1em}$\bigcirc$ & $F_{247}+1$ \\
$125$ & \checkmark\hspace{-1em}$\bigcirc$ & \checkmark\hspace{-1em}$\bigcirc$ & \checkmark\hspace{-1em}$\bigcirc$ & \phantom{\checkmark}\hspace{-1em}$\bigcirc$ & $\boldsymbol{2\cdot G_{125}}$ \\
$126$ & \checkmark & \checkmark & \checkmark\hspace{-1em}$\bigcirc$ & \checkmark\hspace{-1em}$\bigcirc$ & $G_{251}+1$, $2\cdot (G_{125}+1)$ \\
$127$ & \checkmark & \checkmark & \checkmark\hspace{-1em}$\bigcirc$ & \checkmark\hspace{-1em}$\bigcirc$ & $2\cdot G_{127}$ \\
$128$ & \checkmark & \checkmark & \checkmark\hspace{-1em}$\bigcirc$ & \checkmark\hspace{-1em}$\bigcirc$ & $2\cdot (G_{127}+1)$ \\
$129$ & \checkmark & \checkmark & \checkmark\hspace{-1em}$\bigcirc$ & \checkmark\hspace{-1em}$\bigcirc$ & $G_{257}+1$ \\
$130$ & \checkmark & \checkmark & \checkmark\hspace{-1em}$\bigcirc$ & \phantom{\checkmark}\hspace{-1em}$\bigcirc$ & $F_{259}+1$ \\
$131$ & \checkmark & \checkmark & \checkmark\hspace{-1em}$\bigcirc$ & \checkmark\hspace{-1em}$\bigcirc$ & $2\cdot G_{131}$ \\
$132$ & \checkmark & \checkmark & \checkmark\hspace{-1em}$\bigcirc$ & \checkmark\hspace{-1em}$\bigcirc$ & $G_{263}+1$, $2\cdot (G_{131}+1)$ \\
$133$ & \checkmark\hspace{-1em}$\bigcirc$ & \phantom{\checkmark}\hspace{-1em}$\bigcirc$ & \checkmark\hspace{-1em}$\bigcirc$ & \phantom{\checkmark}\hspace{-1em}$\bigcirc$ & \\
$134$ & \checkmark & \checkmark & \checkmark\hspace{-1em}$\bigcirc$ & \phantom{\checkmark}\hspace{-1em}$\bigcirc$ & $F_{267}+1$ \\
$135$ & \checkmark & \checkmark & \checkmark\hspace{-1em}$\bigcirc$ & \checkmark\hspace{-1em}$\bigcirc$ & $G_{269}+1$ \\
$136$ & \checkmark & \checkmark & \checkmark\hspace{-1em}$\bigcirc$ & \checkmark\hspace{-1em}$\bigcirc$ & $G_{271}+1$ \\
$137$ & \checkmark\hspace{-1em}$\bigcirc$ & \checkmark\hspace{-1em}$\bigcirc$ & \checkmark\hspace{-1em}$\bigcirc$ & \checkmark\hspace{-1em}$\bigcirc$ & $\boldsymbol{2\cdot G_{137}}$ \\
$138$ & \checkmark & \checkmark & \checkmark\hspace{-1em}$\bigcirc$ & \checkmark\hspace{-1em}$\bigcirc$ & $2\cdot (G_{137}+1)$ \\
$139$ & \checkmark & \checkmark & \checkmark\hspace{-1em}$\bigcirc$ & \checkmark\hspace{-1em}$\bigcirc$ & $G_{277}+1$, $2\cdot G_{139}$ \\
$140$ & \checkmark & \checkmark & \checkmark\hspace{-1em}$\bigcirc$ & \checkmark\hspace{-1em}$\bigcirc$ & $2\cdot (G_{139}+1)$ \\
$141$ & \checkmark & \checkmark & \checkmark\hspace{-1em}$\bigcirc$ & \checkmark\hspace{-1em}$\bigcirc$ & $G_{281}+1$ \\
$142$ & \checkmark & \checkmark & \checkmark\hspace{-1em}$\bigcirc$ & \checkmark\hspace{-1em}$\bigcirc$ & $G_{283}+1$ \\
$143$ & \checkmark & \checkmark & \checkmark\hspace{-1em}$\bigcirc$ & \checkmark\hspace{-1em}$\bigcirc$ & $2\cdot\operatorname{TPP}(11)$ \\
$144$ & \checkmark & \checkmark & \checkmark\hspace{-1em}$\bigcirc$ & \phantom{\checkmark}\hspace{-1em}$\bigcirc$ & $F_{287}+1$ \\
$145$ & \checkmark & \checkmark & \checkmark\hspace{-1em}$\bigcirc$ & \checkmark\hspace{-1em}$\bigcirc$ & $G_{289}+1$ \\
$146$ & \checkmark & \checkmark & \checkmark\hspace{-1em}$\bigcirc$ & \phantom{\checkmark}\hspace{-1em}$\bigcirc$ & $F_{291}+1$ \\
$147$ & \checkmark & \checkmark & \checkmark\hspace{-1em}$\bigcirc$ & \checkmark\hspace{-1em}$\bigcirc$ & $G_{293}+1$ \\
$148$ & \checkmark & \checkmark & \checkmark\hspace{-1em}$\bigcirc$ & \phantom{\checkmark}\hspace{-1em}$\bigcirc$ & $F_{295}+1$ \\
$149$ & \checkmark\hspace{-1em}$\bigcirc$ & \checkmark\hspace{-1em}$\bigcirc$ & \checkmark\hspace{-1em}$\bigcirc$ & \checkmark\hspace{-1em}$\bigcirc$ & $\boldsymbol{2\cdot G_{149}}$ \\
$150$ & \checkmark & \checkmark & \checkmark\hspace{-1em}$\bigcirc$ & \checkmark\hspace{-1em}$\bigcirc$ & $2\cdot (G_{149}+1)$
\end{tabular}
}
\end{center}
\end{table}

\section*{Acknowledgments}

JWI was supported by NSF DMS
2220301 and a grant from the Simons Foundation. 
JJ was supported by NSF DMS 2220320. 
DGM was supported by NSF DMS 2220304 and an Air Force Summer Faculty Fellowship. 
The views expressed are those of the authors and do not reflect the official guidance or position of the United States Government, the Department of Defense, the United States Air Force, or the United States Space Force.

\appendix

\section{Proof of Proposition~\ref{prop.paley conference matrices are the equivalent}}

For (a), consider the symplectic form on $\mathbb{F}_q^2$ defined by
\[
\big(
~(a,b),~(c,d)~
\big)
\mapsto\operatorname{det}\left(\left[\begin{array}{cc} a&c\\b&d \end{array}\right]\right),
\]
and put $t_\alpha:=(\alpha,1)$ for $\alpha\in\mathbb{F}_q$ and $t_\infty:=(1,0)$.
(One may reindex with $[q+1]$, but we use this indexing for convenience.)
The resulting conference matrix in Example~\ref{ex.symplectic form construction} has entries $C_{i,j}=\chi\left(\operatorname{det}\left(\left[\begin{array}{cc}t_i&t_j\end{array}\right]\right)\right)$ with $i,j\in \mathbb{F}_q\cup\{\infty\}$.
For $\alpha,\beta\in\mathbb{F}_q$, we have
\[
C_{\alpha,\beta}
=\chi(\alpha-\beta)
=\left\{\begin{array}{rl} 
+1&\text{if $\alpha-\beta$ is a quadratic residue}\\
0&\text{if $\alpha=\beta$}\\
-1&\text{else}\end{array}\right.
\]
and
\[
C_{\infty,\infty}
=0,
\qquad
C_{\infty,\beta}
=\chi(1)
=1,
\qquad
C_{\alpha,\infty}
=\chi(-1)
=\left\{\begin{array}{rl}
+1&\text{if $q\equiv 1\bmod 4$}\\
-1&\text{if $q\equiv 3\bmod 4$,}
\end{array}\right.
\]
meaning $C$ is the Paley conference matrix of order $q+1$ (up to reindexing).

For (b), it suffices to show that every conference matrix in Example~\ref{ex.symplectic form construction} is switching equivalent to the Paley conference matrix.
To this end, fix a $2$-dimensional vector space $V$ over $\mathbb{F}_q$ with symplectic form $[\cdot,\cdot]$ and $\tilde{t}_1,\ldots,\tilde{t}_{q+1}\in V$ that span distinct lines in $V$, and let $\tilde{C}$ denote the resulting conference matrix.
First, we recall that all symplectic forms are equivalent, i.e., there exists an invertible linear map $T\colon V\to\mathbb{F}_q^2$ such that $[u,v]=\operatorname{det}\left(\left[\begin{array}{cc}Tu&Tv\end{array}\right]\right)$ for all $u,v\in V$.
Then $T\tilde{t}_1,\ldots,T\tilde{t}_{q+1}\in\mathbb{F}_q^2$ are nonzero vectors that span distinct lines in $\mathbb{F}_q^2$.
Meanwhile, by our proof of (a), the Paley conference matrix $C$ uses a possibly different choice of vectors $\{t_x\}_{x\in\mathbb{F}_q\cup\{\infty\}}$ with the same property.
Thus, there exists a bijection $f\colon[q+1]\to\mathbb{F}_q\cup\{\infty\}$ and scalars $c_1,\ldots,c_{q+1}\in\mathbb{F}_q^\times$ such that $t_{f(i)}=c_i T\tilde{t}_i$.
Then
\begin{align*}
C_{f(i),f(j)}
&=\chi\left(\operatorname{det}\left(\left[\begin{array}{cc}t_{f(i)}&t_{f(j)}\end{array}\right]\right)\right)
=\chi\left(\operatorname{det}\left(\left[\begin{array}{cc}c_i T\tilde{t}_i&c_j T\tilde{t}_j\end{array}\right]\right)\right)
\\
&=\chi\left(c_ic_j\operatorname{det}\left(\left[\begin{array}{cc}T\tilde{t}_i&T\tilde{t}_j\end{array}\right]\right)\right)
=\chi(c_i)\chi(c_j)\chi([\tilde{t}_i,\tilde{t}_j])
=\chi(c_i)\chi(c_j)\tilde{C}_{i,j}.
\end{align*}
That is, $\tilde{C}$ is switching equivalent of $C$.

\end{document}